\documentclass[12pt]{article}

\usepackage{bm}
\usepackage{tipa}
\usepackage{amsthm}
\usepackage{amsmath}
\usepackage{amssymb}
\usepackage{mathtools}
\usepackage{authblk}

\usepackage{natbib}
\usepackage{hyperref}
\usepackage{graphicx}
\usepackage{xr}
\usepackage{url}
\usepackage{enumitem}
\usepackage{nccmath}
\usepackage{geometry}
\usepackage{bbm}

\newcommand{\abs}[1]{\lvert#1\rvert}
\newcommand{\norm}[1]{\lVert#1\rVert}
\DeclarePairedDelimiterX{\inp}[2]{\langle}{\rangle}{#1, #2}

\newcommand{\bE}{\mathbf{E}}
\newcommand{\bV}{\mathbf{Var}}
\newcommand{\bC}{\mathbf{Cov}}
\newcommand{\bP}{\mathbf{P}}

\theoremstyle{plain}
\newtheorem{theorem}{Theorem}
\newtheorem{proposition}{Proposition}
\newtheorem{corollary}{Corollary}
\newtheorem{lemma}{Lemma}

\theoremstyle{definition}

\newtheorem{definition}{Definition}

\theoremstyle{remark}

\graphicspath{
  {./figures/}
}

\begin{document}

\title{Functional limit laws for the intensity measure of point processes and applications}

\author{Giacomo Francisci}
\author{Anand N.\ Vidyashankar}
\affil{}

\newcommand{\address}{{
  \bigskip
  \footnotesize

  G.~Francisci and A.~Vidyashankar, \textsc{Department of Statistics, George Mason University,
  Fairfax, Virginia, 22030}\par\nopagebreak
 \textit{E-mail address}, G.~Francisci: \texttt{gfranci@gmu.edu}, A.~Vidyashankar: \texttt{avidyash@gmu.edu}

}} 

\date{\today}

\maketitle

\begin{abstract}
Motivated by applications to the study of depth functions for tree-indexed
random variables generated by point processes, we describe functional limit theorems for the intensity measure of point processes. Specifically, we establish uniform laws of large numbers and uniform central limit theorems over a class of bounded measurable functions for estimates of the intensity measure. Using these results, we derive the uniform asymptotic properties of half-space depth and, as corollaries, obtain the asymptotic behavior of medians and other quantiles of the standardized intensity measure. Additionally, we obtain uniform concentration upper bound for the estimator of half-space depth. As a consequence of our results, we also derive uniform consistency and uniform asymptotic normality of Lotka-Nagaev and Harris-type estimators for the Laplace transform of the point processes in a branching random walk. \\

\vspace{0.1in} 
\end{abstract}

\section{Introduction}
\label{section:introduction}

Point processes arise in several areas of probability and statistics, including stochastic geometry, renewal and queuing theory, spatial and survival analysis, particle and population processes, and statistical mechanics. Biology, ecology, economics, engineering, epidemiology, finance, geography, image processing, physics, seismology, and stereology are only a few of the many fields of application. In several of these applications, the intensity measure is a key parameter of interest. In this paper, motivated by applications to Cox processes and depth functions for tree-indexed random variables, we study the functional limit laws for the empirical intensity measure.

Let $\{ Y_{i} \}_{i=1}^{\infty}$ be a collection of independent and identically distributed (i.i.d.) point processes on a Polish space $E$ and denote by $\mu_{n} = \frac{1}{n} \sum_{i=1}^{n} Y_{i}$ the empirical measure. The point process $Y_{i}$ can be expressed as $Y_{i} = \sum_{j=1}^{L_{i}} \delta_{X_{i,j}}$, where $\delta_{X_{i,j}}$ is the Dirac measure at a random point $X_{i,j}$ in $E$ and $L_i$ is a positive, integer-valued random variable. The intensity measure $\mu$ of $Y_{1}$ for any Borel set $B$ is given by $\mu(B) = \bE[Y_{1}(B)]$. Under integrability conditions, the law of large numbers yields that
\begin{equation} \label{almost_sure_convergence}
  \mu_{n}(f) \xrightarrow[]{a.s.} \mu(f)
\end{equation}
for any real measurable function $f$ on $E$, where we write $\nu(f)$ for the integral of $f$ with respect to (w.r.t.)\ a measure $\nu$; that is, $\nu(f) = \int f \, d\nu$. In particular, for the intensity measure $\mu$
\begin{equation} \label{intensity_measure}
  \mu(f) = \int f \, d\mu = \bE[Y_{1}(f)] =\bE[\sum_{j=1}^{L_1}f(X_{1,j})].
\end{equation}
In this paper, we study conditions for the convergence in \eqref{almost_sure_convergence} to hold uniformly over a class of functions $\mathcal{F}$ yielding a Glivenko-Cantelli theorem; namely, the uniform law of large numbers (ULLN) over the class $\mathcal{F}$; that is,
\begin{equation} \label{uniform_LLN}
  \sup_{f \in \mathcal{F}} \abs{\mu_{n}(f) - \mu(f)} \xrightarrow[]{a.s.} 0.
\end{equation}
As in the classical empirical processes theory, our sufficient conditions are in terms of random metric entropy, that is, the logarithm of covering numbers of $\mathcal{F}$ w.r.t.\ the $L^{p}$ empirical pseudo-distance $e_{n,p}$ given by
\begin{equation} \label{empirical_Lp_distance}
  e_{n,p}^{p}(f,g) \coloneqq \frac{1}{n} \sum_{i=1}^{n} Y_{i}(\abs{f-g}^{p}), \quad f,g \in \mathcal{F}.
\end{equation}
When $L_{i} \equiv 1$, this reduces to the usual $L^{p}$ empirical distance. We show that in the general case, similar to the case $L_{i} \equiv 1$, the random metric entropy conditions hold when the class $\mathcal{F}$ is VC-subgraph.

Turning to \eqref{uniform_LLN}, a critical issue is that when $\mathcal{F}$ is uncountable, the left-hand side (LHS) is not guaranteed to be measurable. In classical empirical processes theory, these and other measurability issues are typically resolved by requiring random variables to be coordinate projections of an infinite product space, referred to as the canonical probability space. However, such constructions are not immediate in models such as branching processes and branching random walks. Certain suggestions, such as the use of outer expectation and outer probability, also lead to technical difficulties. Specifically, the classical Fubini's theorem will not hold under outer expectations; indeed, one only obtains that the repeated marginal outer expectations do not exceed the joint outer expectations. It is pertinent to note that the proofs of ULLN require Fubini's theorem when using symmetrization techniques.

In the i.i.d.\ setting, a typical proof of ULLN involves first establishing the convergence in probability and then verifying its equivalence to almost sure convergence via reverse martingale arguments as in \citet{Nolan-1987} and \citet{Strobl-1995}. The existence of canonical probability space is used in this argument. In the current paper, we adopt a different route and refine the techniques in \citet{Kuelbs-1979} and \citet{deAcosta-1981} that are developed for Banach-valued random variables measurable with respect to the Borel $\sigma$-field generated by the norm.

The second contribution of the paper concerns the central limit theorem (CLT) for the centered and scaled process $W_{n} = \sqrt{n}(\mu_{n}-\mu)$. We assume that the class $\mathcal{F}$ is uniformly bounded and hence $W_{n}$ belongs to $\ell_{\infty}(\mathcal{F})$, the space of bounded functionals on $\mathcal{F}$. We show that $W_{n}$ converges in distribution to $W$, a centered Gaussian process with covariance function $\gamma$, in the sense of \citet{Hoffmann-1974}, under random metric entropy conditions, and 
\begin{equation*}
  \gamma(f,g) \coloneqq \bE[(Y_{1}(f)-\mu(f))(Y_{1}(g)-\mu(g))]
\end{equation*}  
for all $f,g \in \mathcal{F}$. That is, the class $\mathcal{F}$ is Donsker \citep{Gine-1984, Gine-2016}. The sufficient conditions for the above central limit theorem are again given in terms of random metric entropy and hold when the class $\mathcal{F}$ is VC-subgraph. Next, when $\mathcal{F}$ is a collection of indicators of VC-classes of sets, we establish an upper bound on the deviation probabilities of the LHS of \eqref{uniform_LLN}. In particular, our results hold when the above sets are half-spaces in $\mathbb{R}^{d}$, which we exploit to study asymptotic properties of the half-space depth of the intensity measure of the point process in a variety of probability models. Applying the results to branching random walk models, one obtains uniform consistency and central limit theorems for estimates of the Laplace transform of the point process \citep{Biggins-1992}.

The rest of the paper is organized as follows. Section \ref{section:main_results} contains a precise statement of the paper's main results. Section \ref{section:applications} contains applications to depth functions and branching random walk. Section \ref{section:preliminary_results} collects a series of preliminary results involving symmetrization by Rademacher random variables. Detailed proofs of the results are given in Sections \ref{section:proofs_of_main_results} and \ref{section:proofs_of_other_results}. The paper ends with some concluding remarks in Section \ref{section:concluding_remarks}.

\section{Main results}
\label{section:main_results}

Let $\{ Y_{i} \}_{i=1}^{\infty}$ denote a collection of i.i.d.\ point processes on the probability space $(\Omega,\Sigma,\bP)$, where as before $Y_{i} = \sum_{j=1}^{L_{i}} \delta_{X_{i,j}}$, $\{ L_{i} \}_{i=1}^{\infty}$ are $\mathbb{N}$-valued random variables with distribution $p_{j} = \bP(L_{i}=j)$, and $\delta_{X_{i,j}}$ is the Dirac measure at the $E$-valued random variable $X_{i,j}$. Throughout, we assume that $E$ is a Polish space. We refer to Section \ref{section:preliminary_results} for more details on the probability space and \citet{Kallenberg-2017} for additional details on point processes. As described in the introduction, we study the convergence of the empirical process $\mu_{n} - \mu = \{ \mu_{n}(f) - \mu(f) \}_{f \in \mathcal{F}}$, where
\begin{equation*}
  \mu_{n}(f) = \frac{1}{n} \sum_{i=1}^{n} Y_{i}(f), \text{ } Y_{i}(f) = \sum_{j=1}^{L_{i}} f(X_{i,j}), \text{ and } \mu(f) =  \bE[Y_{1}(f)].
\end{equation*}
If the class $\mathcal{F}$ is uniformly bounded, then $\mu_{n} - \mu$ has bounded sample paths, that is, $\mu_{n} - \mu$ is a random variable taking values on $(\ell_{\infty}(\mathcal{F}), C_{\infty}(\mathcal{F}))$, where $\ell_{\infty}(\mathcal{F})$ is the space of bounded functionals on $\mathcal{F}$ and $C_{\infty}(\mathcal{F})$ is the cylindrical $\sigma$-algebra (the coarsest $\sigma$-algebra containing the cylinder sets) on $\ell_{\infty}(\mathcal{F})$. Similarly, $\{ Y_{i} \}_{i=1}^{\infty}$ is a sequence of i.i.d.\ $(\ell_{\infty}(\mathcal{F}), C_{\infty}(\mathcal{F}))$-valued random variables. The space $\ell_{\infty}(\mathcal{F})$ is endowed with the supremum norm
\begin{equation*}
  \norm{H} \coloneqq \norm{H}_{\mathcal{F}} \coloneqq \sup_{f \in \mathcal{F}} \abs{H(f)},
\end{equation*}
where $H \in \ell_{\infty}(\mathcal{F})$. However, the $\sigma$-algebra $C_{\infty}(\mathcal{F})$ may not contain the open and closed balls induced by $\norm{\cdot}$. We need the following definition of measurability (see Definition 3.7.11 of \citet{Gine-2016}).
\begin{definition} \label{definition:measurable_class}
A class of functions $\mathcal{F}$ is measurable if for each $a_{1}, \dots, a_{n}, b \in \mathbb{R}$ and $n \in \mathbb{N}$, the quantity $\norm{ \sum_{i=1}^{n} a_{i} Y_{i} + b \mu}$ is measurable w.r.t.\ the completion of $(\Omega,\Sigma,\bP)$.
\end{definition}
It is worth noting that a sufficient condition for measurability in Definition \ref{definition:measurable_class} is that every function $f \in \mathcal{F}$ is a pointwise limit of functions $f_{k}$ in a countable subclass $\mathcal{F}_{0}$ of $\mathcal{F}$. In particular, this holds if $\mathcal{F}$ itself is countable. We make the following assumptions throughout the paper.
\begin{enumerate}[label=(\textbf{H\arabic*})]
\item $\bE[L_{1}^{2}] < \infty$, and \label{H1}
\item $\mathcal{F}$ is a uniformly bounded non-empty measurable class of real functions on $E$.  \label{H2}
\end{enumerate}
$\mathcal{F}$ being uniformly bounded entails that $\sup_{f \in \mathcal{F}} \abs{f(x)} \leq M$ for some constant $M$ and all $x \in E$. We frequently make use of the $L^{p}$ empirical pseudo-distance on $\mathcal{F}$ defined in \eqref{empirical_Lp_distance}. For any $\epsilon>0$ the covering number of a pseudo-metric space $(T, e)$ is
\begin{equation*}
  N(T, e, \epsilon) \coloneqq \inf \{N : \exists~ t_{1}, \dots, t_{N} \in T \text{ such that } \min_{i=1,\dots,N} e(t_{i},t) \leq \epsilon \text{ for all } t \in T \}.
\end{equation*}
Our first result is the uniform law of large numbers for the intensity measure $\mu$ in \eqref{intensity_measure}. We denote by $\xrightarrow[]{p^{\ast}}$ convergence in outer probability and use the superscript $^{\ast}$ on a random quantity to denote its $\bP$-measurable cover (see Section \ref{section:preliminary_results} for more details).

\begin{theorem} \label{theorem:uniform_LLN}
  Assume \ref{H1}-\ref{H2}. Then, $\norm{\mu_{n} - \mu } \xrightarrow[]{a.s.} 0$ if one of the following conditions hold: \\
  (i) for all $\epsilon>0$ and some $p \geq 1$ $\frac{1}{n} \log(N(\mathcal{F}, e_{n,p}, \epsilon)) \xrightarrow[]{p^{\ast}} 0$, or \\
  (ii) for all $\delta>0$
\begin{equation*}
  \lim_{n \to \infty} \bE[\min(1, \frac{1}{\sqrt{n}} \int_{0}^{\delta} \sqrt{\log(N^{\ast}(\mathcal{F}, e_{n,2},\epsilon))} \, d \epsilon )] = 0.
\end{equation*}
\end{theorem}
The structure of the proof of Theorem \ref{theorem:uniform_LLN} is similar to the case $L_{i} \equiv 1$ but involves significant technical differences. The method of proof involves first establishing that the conditions (i) and (ii) yield convergence in probability of $\norm{\mu_{n} - \mu }$. In the next step, we prove that convergence in probability implies almost sure convergence. As explained previously, if the space is canonical, this follows from a martingale argument. Our methods are different. A detailed proof is contained in Section \ref{section:proofs_of_main_results}. Section \ref{section:preliminary_results} includes some preliminary results. This section also contains remarks on the underlying probability space and additional random variables used.

We now turn our attention to the Glivenko-Cantelli property of VC-subgraph classes of functions. We begin with a precise definition of VC-subgraph class (Definition 3.6.8 of \citet{Gine-2016}). We recall that the subgraph of a function $f: E \to \mathbb{R}$ is
\begin{equation*}
  G_{f} = \{ (x,y) \in E \times \mathbb{R} : x \in E, \, y \in \mathbb{R}, \text{ and } y \leq f(x) \}.
\end{equation*}
The family of subgraphs is denoted by $\mathcal{C} = \mathcal{C}_{\mathcal{F}} = \{ G_{f} : f \in \mathcal{F} \}$. The class $\mathcal{F}$ is VC-subgraph if
\begin{equation*}
  \mathbf{v} = \mathbf{v}(\mathcal{C}) = \inf \{ k \geq 0 \, : \, m^{\mathcal{C}}(k) < 2^{k} \} < \infty,
\end{equation*}
where $m^{\mathcal{C}}(k) = \sup_{A \subset E \times \mathbb{R} \, : \, \abs{A} = k} \Delta^{\mathcal{C}} (A)$ and $\Delta^{\mathcal{C}} (A) = \abs{ \{ A \cap C : C \in \mathcal{C} \} }$ is the cardinality of the collection of all sets obtained by intersection of $A$ with sets $C \in \mathcal{C}$. The quantity $\mathbf{v}(\mathcal{C})$ is referred to as VC-dimension of $\mathcal{F}$. We make the following assumption.
\begin{enumerate}[label=(\textbf{H\arabic*})] \setcounter{enumi}{2}
\item $\mathcal{F}$ is a VC-subgraph class, that is, it has finite VC-dimension $\mathbf{v}$. \label{H3}
\end{enumerate}
It is easy to see that the VC-dimension $\mathbf{v}(\mathcal{C}_{\mathcal{F}})$ of a class $\mathcal{F} = \{ \mathbf{I}_{D} \, : \, D \in \mathcal{D} \}$ of indicator functions equals to the VC-dimension $\mathbf{v}(\mathcal{D})$ of the class of sets $\mathcal{D}$ (see Section 3.6.2 of \citet{Gine-2016}). The next proposition provides an upper bound on the covering number of VC-subgraph classes of functions. We let $S_{n} \coloneqq \sum_{i=1}^{n} L_{i}$ be the total number of random variables in the first $n$ point processes.

\begin{proposition} \label{proposition:VC-subgraph_bound}
  Assume \ref{H2}-\ref{H3}. Then, for some constant $c_{\mathbf{v}}$ depending only on the VC-dimension $\mathbf{v}$ of the class $\mathcal{F}$, it holds that
\begin{equation} \label{VC-subgraph_bound}
  N(\mathcal{F}, e_{n,p}, \epsilon) \leq \max \biggl( c_{\mathbf{v}}, \biggl(4 \cdot \frac{S_{n}}{n} \cdot \biggl(\frac{2M}{\epsilon}\biggr)^{p} \biggr)^{\mathbf{v}} \biggr).
\end{equation}    
\end{proposition}
By the law of large numbers $\frac{S_{n}}{n} \xrightarrow[]{a.s.} \bE[L_{1}]$. An immediate consequence of Theorem \ref{theorem:uniform_LLN} and Proposition \ref{proposition:VC-subgraph_bound} with $p=1$ is the following Glivenko-Cantelli theorem for the empirical process $\{ \mu_{n}(f) \}_{f \in \mathcal{F}}$.
\begin{corollary} \label{corollary:uniform_LLN}
Assume \ref{H1}-\ref{H3}. Then $\norm{\mu_{n} - \mu } \xrightarrow[]{a.s.} 0$.
\end{corollary}
Next, we turn to convergence in distribution of the empirical process $W_{n} \coloneqq \sqrt{n} (\mu_{n} - \mu)$. We say that a stochastic process $\{ W_{n}(f) \}_{f \in \mathcal{F}}$ converges in law to a process $\{ W(f) \}_{f \in \mathcal{F}}$ with tight Borel law, in symbols, $W_{n} \xrightarrow[]{d} W \text{ in } \ell_{\infty}(\mathcal{F})$, if
\begin{equation*}
  \lim_{n \to \infty} \bE^{\ast}[H(W_{n})] = \bE[H(W)]
\end{equation*}
for all bounded and continuous functions $H: \ell_{\infty}(\mathcal{F}) \to \mathbb{R}$, where $\bE^{\ast}$ denotes outer expectation. For $\delta \in [0,\infty]$ and $p \geq 1$, we define
\begin{equation*}
  \mathcal{F}_{\delta,p}^{\prime} \coloneqq \{ (f-g)^{p} : f,g \in \mathcal{F} \text{ and } \norm{f-g}_{L^{2}(\mu)} \leq \delta \},
\end{equation*}
where $\norm{\cdot}_{L^{p}(\mu)}$ is the $L^{p}(\mu)$ norm on $\mathcal{F}$ given by $\norm{f}_{L^{p}(\mu)}^{p} \coloneqq \mu(\abs{f}^{p})$, $f \in \mathcal{F}$. We have the following central limit theorem for the empirical process $\mu_{n}$.

\begin{theorem} \label{theorem:uniform_CLT}
  Assume \ref{H1}-\ref{H2} and that the classes of functions $\mathcal{F}^{\prime}_{\infty,2}$ and $\{ \mathcal{F}_{\delta,1}^{\prime} \}_{\delta>0}$ are measurable. If
\begin{equation} \label{assumption_uniform_CLT}
  \lim_{\delta \to 0^{+}} \limsup_{n \to \infty} \bE[\min(1, \int_{0}^{\delta} \sqrt{\log(N^{\ast}(\mathcal{F},e_{n,2},\epsilon))} \, d\epsilon)] = 0,
\end{equation}
then  
\begin{equation*}
  \sqrt{n} (\mu_{n} - \mu) \xrightarrow[]{d} W \text{ in } \ell_{\infty}(\mathcal{F}),
\end{equation*}
where $W$ is a Gaussian process with covariance function
\begin{equation*}
  \bC[W(f),W(g)] = \gamma(f,g) = \bE[(Y_{1}(f)-\mu(f))(Y_{1}(g)-\mu(g))].
\end{equation*}  
\end{theorem}
The proof of Theorem \ref{theorem:uniform_CLT} is contained in Section \ref{section:proofs_of_main_results}. When Theorem \ref{theorem:uniform_CLT} holds, that is, when $\mu_{n}$ satisfies a uniform CLT over $\mathcal{F}$, the class $\mathcal{F}$ is said to be $\mu$-Donsker or Donsker. As \eqref{assumption_uniform_CLT} holds under \ref{H3} and a moment assumption on $L_{i}$, VC-subgraph classes of functions are Donsker.
\begin{corollary} \label{corollary:uniform_CLT}
  Assume \ref{H1}-\ref{H3} and that the classes of functions $\mathcal{F}^{\prime}_{\infty,2}$ and $\{ \mathcal{F}_{\delta,1}^{\prime} \}_{\delta>0}$ are measurable. Then
\begin{equation*}
  \sqrt{n} (\mu_{n} - \mu) \xrightarrow[]{d} W \text{ in } \ell_{\infty}(\mathcal{F}).
\end{equation*}
\end{corollary}
Our next result is motivated by an application to the rate of convergence of \emph{half-space depth of the intensity measure}. Thus, we now consider a class of functions $\mathcal{F} = \{ \mathbf{I}_{D} \, : \, D \in \mathcal{D} \}$ given by indicator functions of a class of sets $\mathcal{D}$. We now establish a uniform inequality for deviations from the intensity measure. We let $S_{n,p} \coloneqq \sum_{i=1}^{n} L_{i}^{p}$, $p \geq 1$. In particular, $S_{n,1}=S_{n}$.

\begin{theorem} \label{theorem:uniform_rates_of_convergence}
  Assume \ref{H1}-\ref{H3} and let $\mathcal{F} = \{ \mathbf{I}_{D} \, : \, D \in \mathcal{D} \}$. For all $\epsilon>0$ and $n \geq 8 \cdot \bE[L_{1}^{2}]/\epsilon^{2}$ it holds that
\begin{equation} \label{uniform_rates_of_convergence}
  \bP(\norm{\mu_{n}-\mu} \geq \epsilon) \leq 8 \cdot \bE \biggl[ m^{\mathcal{D}}(S_{n}) \cdot \exp \biggl( - \frac{\epsilon^{2}}{2^{5}} \cdot \frac{n^{2}}{S_{n,2}} \biggr) \biggr].
\end{equation}
In particular, for all $\alpha,\beta>0$
\begin{equation}  \label{uniform_rates_of_convergence_2}
  \bP(\norm{\mu_{n}-\mu} \geq \epsilon) \leq 16 \cdot (\alpha n)^{\mathbf{v}-1} \cdot \exp \biggl( -\frac{\epsilon^{2}}{2^{5}} \cdot \frac{n}{\beta} \biggr) + \bP(S_{n} > \alpha n) + \bP(S_{n,2} > \beta n).
\end{equation}    
\end{theorem}
The proof of Theorem \ref{theorem:uniform_rates_of_convergence} is contained in Section \ref{section:proofs_of_main_results}. If $L_{i} \equiv 1$ then \eqref{uniform_rates_of_convergence} reads
\begin{equation*}
  \bP(\norm{\mu_{n}-\mu} \geq \epsilon) \leq 8 \cdot m^{\mathcal{D}}(n) \cdot \exp \biggl( - \frac{\epsilon^{2} n}{2^{5}} \biggr)
\end{equation*}
(see also Inequality 1 in Chapter 26 of \citet{Shorack-1986} and (2.4) of \citet{Pollard-1981}). We also notice that by \citet{Sauer-1972}'s lemma, $m^{\mathcal{D}}(n) \leq 2 n^{\mathbf{v}-1}$ for all $n \geq 1$, which can be used to obtain an upper bound on the expectation in \eqref{uniform_rates_of_convergence}. By the law of large numbers, the quantities $\bP(S_{n} > \alpha n)$ and $\bP(S_{n,2} > \beta n)$ converge to zero whenever $\alpha>\bE[L_{1}]$ and $\beta>\bE[L_{1}^{2}]$. Explicit upper bounds on the rate of convergence follow from additional moment conditions on $L_{i}$. For instance, if $\bE[\exp(\theta L_{1}^{2})])<\infty$ for some $\theta>0$, then Chernoff bound for $S_{n}$ yields
\begin{equation*}
  (\bP(S_{n} > \alpha n) )^{1/n} \leq \inf_{\theta >0} (\exp(-\theta \alpha) \bE[\exp(\theta L_{1})]).
\end{equation*}
Similarly, in the Chernoff's bound for $S_{n,2}$, $L_{1}$ is replaced by $L_{1}^{2}$. These upper bounds yield an exponential decay for $\bP(\norm{\mu_{n}-\mu} \geq \epsilon)$. We notice that large values of $\mathbf{v}$ in \eqref{uniform_rates_of_convergence_2} are related to the size of $\mathcal{F}$ being ``large''. For the class of half-spaces in $\mathbb{R}^{d}$, $\mathbf{v}$ equals $d+1$ (see Proposition \ref{proposition:halfspace_depth} below). 

\section{Applications} \label{section:applications}

In this section, we use the functional limit laws described above to establish properties of depth functions for intensity measures. These, in turn, allow us to study the quantiles of the tree-indexed random variables. Additionally, the methods allow for establishing uniform consistency and asymptotic normality results for estimates of the Laplace transform of the point process, one of the critical components in the study of branching random walk. We begin with the half-space depth. 

\subsection{Depth functions} \label{subsection:depth_functions}

In this subsection, we take $E=\mathbb{R}^{d}$. We study the asymptotic properties of the half-space depth \citep{Zuo-2000a}. We denote by $\mathcal{H} \coloneqq \{ H_{x,u} \, : \, x \in \mathbb{R}^{d}, \, u \in S^{d-1} \}$ the class of closed half-spaces in $\mathbb{R}^{d}$, where $H_{x,u}=\{ y \in \mathbb{R}^{d} \, : \, \inp{y}{u} \leq \inp{x}{u} \}$ is the closed half-space with boundary point $x$ and outer normal $u$. The half-space depth of $x \in \mathbb{R}^{d}$ w.r.t.\ $\mu$ is defined by
\begin{equation} \label{halfspace_depth}
  D(x,\mu) \coloneqq \inf_{u \in S^{d-1}} \mu(H_{x,u}).
\end{equation}
We notice that in our case $\mu$ is a finite measure. We make the following assumptions on $\mu$ (Assumption (S) of \citet{Masse-2004}).
\begin{enumerate}[label=(\textbf{H\arabic*})] \setcounter{enumi}{3}
\item $\mu(\partial H)=0$ for all $H \in \mathcal{H}$. \label{H4}
\end{enumerate}
It follows from Lemma \ref{lemma:continuity_halfspace_depth} in Section \ref{section:proofs_of_main_results} that under assumption \ref{H4} the infimum in \eqref{halfspace_depth} is achieved, that is, $D(x,\mu) = \mu(H_{x,u})$ for some $u \in S^{d-1}$. We restrict our attention to subsets of $\mathbb{R}^{d}$ for which the minimal direction is unique. Accordingly, we consider the set
\begin{equation*}
  \mathcal{R}(\mu) \coloneqq \{ x \in \mathbb{R}^{d} \, : \, \exists ! \, u_{x} \in S^{d-1} \text{ such that } D(x,\mu) = \mu(H_{x,u_{x}}) \}.
\end{equation*}
This assumption may be relaxed by requiring that the minimizer set is either finite or $S^{d-1}$ (see Assumption (LR) of \citet{Masse-2004}). However, in the latter cases the limiting process in Proposition \ref{proposition:halfspace_depth} (ii) below may not be Gaussian. We critically use that the class of indicator functions $\mathcal{F} = \{ \mathbf{I}_{H} \, : \, H \in \mathcal{H} \}$ has VC-dimension $\mathbf{v}=d+1$.
\begin{proposition} \label{proposition:halfspace_depth}
  Assume \ref{H1}. The following holds:
\begin{flalign*}
  \text{(i) } \sup_{x \in \mathbb{R}^{d}} \abs{D(x,\mu) - D(x,\mu_{n})} \xrightarrow[]{a.s.} 0. &&
\end{flalign*}  
(ii) If \ref{H4} holds true and $A \neq \emptyset$ is a closed subset of $\mathcal{R}(\mu)$, then
\begin{equation*}
  \sqrt{n} (D(\cdot,\mu) - D(\cdot,\mu_{n})) \xrightarrow[]{d} W \text{ in } \ell_{\infty}(A),
\end{equation*}
where $W$ is a Gaussian process with covariance function
\begin{equation*}
  \bC[W(x),W(y)] = \bE[(Y_{1}(H_{x,u_{x}})-\mu(H_{x,u_{x}}))(Y_{1}(H_{y,u_{y}})-\mu(H_{y,u_{y}}))].
\end{equation*} 
(iii) For all $\alpha,\beta,\epsilon>0$ and $n \geq 8 \cdot \bE[L_{1}^{2}]/\epsilon^{2}$
\begin{equation*}
  \bP(\sup_{x \in \mathbb{R}^{d}} \abs{D(x,\mu) - D(x,\mu_{n})} \geq \epsilon) \leq 16 \cdot (\alpha n)^{d} \cdot \exp \biggl( -\frac{\epsilon^{2}}{2^{5}} \cdot \frac{n}{\beta} \biggr) + \bP(S_{n} > \alpha n) + \bP(S_{n,2} > \beta n).
\end{equation*}  
\end{proposition}
We notice that in the case $L_{i} \equiv 1$ the factor in the exponential can be improved under additional continuity and exponential decay assumptions on the underlying random variables (see \citet{Burr-2017}). Part (i) and (iii) are an immediate consequence of Corollary \ref{corollary:uniform_LLN} and Theorem \ref{theorem:uniform_rates_of_convergence}. Part (ii) is based on Corollary \ref{corollary:uniform_CLT} and Lemma 5.4 of \citet{Masse-2004}. The details are in Section \ref{section:proofs_of_other_results}. A specifically interesting example is when the point processes $\{Y_{i}\}_{i=1}^{\infty}$ arise from the branching random walk (see \eqref{point_processes_branching_random_walk} and Lemma \ref{lemma:independent_and_identically_distributed} in Subsection \ref{subsection:tree-indexed_random_variables} below). In this case, Proposition \ref{proposition:halfspace_depth} yields inference for depth function for tree-indexed random variables and can be used to obtain asymptotic behavior of medians and other quantiles of the standardized intensity measure \citep{Zuo-2000b}.

\subsection{Tree-indexed random variables} \label{subsection:tree-indexed_random_variables}

Using Ulam-Harris notation we write $\mathbb{V} = \{ \emptyset \} \cup \cup_{j=1}^{\infty} \mathbb{N}^{j}$ for the set of all potential vertices in the tree and consider a sequence of i.i.d.\ point processes $\{ Y_{v} \}_{v \in \mathbb{V}}$. The tree starts with a single vertex $\emptyset$ at time $0$ and evolves as follows: each vertex $v$ at time $j$ yields new vertices $(v,l)$ and associated random variables $X_{(v,l)}$ according to the point process $Y_{v} = \sum_{l=1}^{L_{v}} \delta_{X_{(v,l)}}$. This gives rise to a Galton-Watson tree with $E$-valued random variables attached to each vertex. The set of actual vertices in the Galton-Watson tree is denoted by $V \subset \mathbb{V}$. We order the vertices with respect to (w.r.t.) the breadth-first order induced by Ulam-Harris notation. Thus, we write $V=\{ v_{1}, v_{2}, \dots \}$, where $v_{1}=\emptyset$ etc. The above construction yields a Galton-Watson process $\{ \abs{V_{j}}\}_{j=0}^{\infty}$, where $V_{j}$ is the set of vertices at time (generation) $j$ and $\abs{V_{j}}$ is the cardinality of $V_{j}$ (the number of edges emanating from $V_{j}$). By virtue of the following lemma, the results of this paper hold for the point processes
\begin{equation} \label{point_processes_branching_random_walk}
  Y_{i} \coloneqq Y_{v_{i}} = \sum_{w \in \mathbb{V}} Y_{w} \, \delta_{w}(v_{i}).
\end{equation}
  We assume that $\bP(\abs{V_{1}}=0)=0$ which ensures non-extinction of the Galton-Watson process. In particular, each vertex $v_{i}$ arises from one of the point processes $Y_{v_{1}}, \dots, Y_{v_{i-1}}$ assuring that $Y_{v_{i}}$ is well-defined.
\begin{lemma} \label{lemma:independent_and_identically_distributed}
  Assume that $\bP(\abs{V_{1}}=0)=0$. Then, $\{ Y_{v_{i}} \}_{i=1}^{\infty}$ is a sequence of i.i.d.\ $(\ell_{\infty}(\mathcal{F}), C_{\infty}(\mathcal{F}))$-valued random variables.
\end{lemma}
An interesting model studied in the literature is the branching random walk. In this case, $E$ is also a group (the operation is referred to as addition in the following) and the random variables are defined recursively by adding the new random variables to the random variable of the previous generation. Specifically, the branching random walk starts at time $0$ at $Z_{\emptyset}=0$. The positions $\{Z_{w} : w \in V_{j+1} \}=\{ Z_{(v,l)} : v \in V_{j}, \, l = 1, \dots,L_{v} \}$ at time $j+1$ are obtained recursively from the positions of the previous generation by $Z_{(v,l)}=Z_{v}+X_{(v,l)}$, where $v \in V_{j}$ and $X_{(v,l)}$ arises from the point process $Y_{v} = \sum_{l=1}^{L_{v}} \delta_{X_{(v,l)}}$. On the other hand, we notice that the point processes $\{ Y_{v} \}_{v \in V}$ can be reconstructed from the set of positions $\{ Z_{v} : v \in V \}$ and the genealogical structure $V$ of the tree. Indeed, for each vertex $v \in V$ we have that $Y_{v}=\sum_{l=1}^{L_{v}} \delta_{Z_{(v,l)}-Z_{v}}$. 

We begin by describing the limiting results concerning the intensity measure of the point processes on the tree. The Lotka-Nagaev estimator $\hat{\mu}_{j}$ of $\mu$ is given by
\begin{equation*}
  \hat{\mu}_{j}(f) = \frac{1}{\abs{V_{j}}} \sum_{v \in V_{j}} Y_{v}(f).
\end{equation*}
Alternatively, a Harris-type estimator $\tilde{\mu}_{j}$ of $\mu$ is given by
\begin{equation*}
  \tilde{\mu}_{j}(f) = \frac{1}{\sum_{l=0}^{j}\abs{V_{l}}} \sum_{i=1}^{\sum_{l=0}^{j}\abs{V_{l}}} Y_{v_{i}}(f).
\end{equation*}
It is worth noticing here that Harris-type estimator uses information for the entire tree up to generation $j+1$, while the Lotka-Nagaev estimator uses the genealogical structure of the vertices in $V_{j}$ and $V_{j+1}$. When $f \equiv 1$ one obtains $Y_{v_{i}}(1)=L_{v_{i}}$ and the estimators reduce to the classical estimators of the mean of a supercritical branching process, that is,
\begin{equation*}
  \hat{\mu}_{j}(1)= \frac{\abs{V_{j+1}}}{\abs{V_{j}}} \text{ and } \tilde{\mu}_{j}(1) = \frac{\sum_{l=1}^{j+1}\abs{V_{l}}}{\sum_{l=0}^{j}\abs{V_{l}}}.
\end{equation*}
We also let 
\begin{equation*}
  \hat{W}_{j} \coloneqq \abs{V_{j}}^{1/2} (\hat{\mu}_{j} - \mu) \text{ and } \tilde{W}_{j} \coloneqq (\sum_{l=0}^{j}\abs{V_{l}})^{1/2} (\tilde{\mu}_{j} - \mu).
\end{equation*}

\begin{proposition} \label{proposition:Lotka-Nagaev-estimator}
  Assume \eqref{H1}-\eqref{H3}, $\bP(\abs{V_{1}}=0)=0$, and $\bE[L_{1}]>1$. The following holds: \\
(i) $\sup_{f \in \mathcal{F}} \abs{\hat{\mu}_{j}(f) - \mu(f)} \xrightarrow[]{a.s.} 0$, and \\
(ii) if $\mathcal{F}^{\prime}_{\infty,2}$ and $\{ \mathcal{F}_{\delta,1}^{\prime} \}_{\delta>0}$ are measurable, then $\hat{W}_{j} \xrightarrow[]{d} W$, where $W$ is defined in Theorem \ref{theorem:uniform_CLT}.
\end{proposition}

\noindent Our next proposition is concerned with the Harris-type estimator.

\begin{proposition} \label{proposition:Harris-estimator}
  Assume \eqref{H1}-\eqref{H3} and $\bP(\abs{V_{1}}=0)=0$. The following holds: \\
(i) $\sup_{f \in \mathcal{F}} \abs{\tilde{\mu}_{j}(f) - \mu(f)} \xrightarrow[]{a.s.} 0$, and \\
(ii) if $\mathcal{F}^{\prime}_{\infty,2}$ and $\{ \mathcal{F}_{\delta,1}^{\prime} \}_{\delta>0}$ are measurable, then $\tilde{W}_{j} \xrightarrow[]{d} W$.
\end{proposition}
It is worth noticing that Proposition \ref{proposition:Lotka-Nagaev-estimator} uses Corollaries \ref{corollary:uniform_LLN} and \ref{corollary:uniform_CLT} for deterministic $n$, whereas Part (ii) of Proposition \ref{proposition:Harris-estimator} uses an extension of Theorem \ref{theorem:uniform_CLT} for random $n$. The latter provides a uniform version of Theorem 3 of \citet{Kuelbs-2011} with $b_{j} = 1$ (defined there). A specifically interesting example, for both propositions, is when $E=[a,b]$ and $\mathcal{F}=\{ e^{\theta (\cdot)} \, : \,  \abs{\theta} \leq R \}$ for some $R>0$. We notice that this class is VC-subgraph of index $\mathbf{v}=3$ (cfr.\ Proposition 3.6.12 of \citet{Gine-2016}). Another example is when $E=\mathbb{R}_{+}$, $X_{w}$ is the lifetime of vertex $w \in V$ in Bellman-Harris process, and $\mathcal{F}=\{ e^{\theta (\cdot)} \, : \,  \theta \leq 0 \}$. In both the cases, Propositions \ref{proposition:Lotka-Nagaev-estimator} and \ref{proposition:Harris-estimator} yield consistency and asymptotic normality for the Lotka-Nagaev and Harris-type estimators $\hat{m}_{j}(\cdot)$ and $\tilde{m}_{j}(\cdot)$ of $m(\cdot)$, where
\begin{align*}
  \hat{m}_{j}(\theta)& \coloneqq \hat{\mu}_{j}(e^{\theta (\cdot)}) = \frac{1}{\abs{V_{j}}} \sum_{v \in V_{j}} \sum_{l=1}^{L_{v}} e^{\theta X_{(v,l)}}, \\
  \tilde{m}_{j}(\theta)& \coloneqq \tilde{\mu}_{j}(e^{\theta (\cdot)}) = \frac{1}{\sum_{l=0}^{j}\abs{V_{l}}} \sum_{i=1}^{\sum_{l=0}^{j}\abs{V_{l}}} \sum_{l=1}^{L_{v_{i}}} e^{\theta X_{(v_{i},l)}}, \text{ and} \\
  m(\theta) &\coloneqq \mu(e^{\theta (\cdot)}) = \bE[\sum_{l=1}^{L_{\emptyset}} e^{\theta X_{l}}].
\end{align*}
We conclude this subsection by studying the convergence of $\hat{W}_{j}$ over multiple generations $j$. To this end, we need an additional assumption ensuring the Borel measurability of $\{ \hat{W}_{j} \}_{j=1}^{\infty}$. A sufficient condition for this is that the class $\mathcal{F}$ is finite.
\begin{proposition} \label{proposition:Lotka-Nagaev-estimator_joint_convergence}
Assume \eqref{H1}, $\bP(\abs{V_{1}}=0)=0$, and $\bE[L_{1}]>1$. If $\mathcal{F}$ is a non-empty and finite class of bounded functions, then for all $s \in \mathbb{N}$, as $j \rightarrow \infty$
\begin{equation*}
  (\hat{W}_{j+1}, \dots, \hat{W}_{j+s}) \xrightarrow[]{d} (W^{1},\dots,W^{s}),
\end{equation*}    
where $\{W^{l}\}_{l=1}^{s}$ are independent copies of the random variable $W$ in Proposition \ref{proposition:Lotka-Nagaev-estimator}.
\end{proposition}
\noindent By embedding $(\hat{W}_{j+1}, \dots, \hat{W}_{j+s})$ into $(\ell_{\infty}(\mathcal{F}))^{\infty}$ and using the metric
\begin{equation*}
  d(\mathbf{s}, \mathbf{t}) = \sum_{k=1}^{\infty} \frac{1}{2^{k}} \frac{\norm{s_{k}-t_{k}}}{1+\norm{s_{k}-t_{k}}},
\end{equation*}
one can strengthen Proposition \ref{proposition:Lotka-Nagaev-estimator_joint_convergence} to obtain the convergence 
\begin{equation*}
  (\hat{W}_{r_{n}+1}, \dots, \hat{W}_{n}, 0, 0, \dots) \xrightarrow[]{d} (W^{1},W^{2},\dots),
\end{equation*}
where $0 \leq r_{n} \leq n-1$ with $r_{n}, n-r_{n} \to \infty$. This follows from the arguments in Theorem 1 of \citet{Kuelbs-2011}.
\subsection{Mixed binomial point processes} \label{subsection:mixed_binomial_point_processes}
In this subsection, we make the following additional assumption on $\{ Y_{i} \}_{i=1}^{\infty}$:
\begin{enumerate}[label=(\textbf{H\arabic*})] \setcounter{enumi}{4}
\item $X_{i,j}$ are i.i.d.\ and independent of $L_{i}$. \label{H5}
\end{enumerate}
The point processes $\{ Y_{i} \}_{i=1}^{\infty}$ are called \emph{mixed binomial}. Additionally, if $L_{i}$ is constant almost surely, then they are called \emph{binomial} as $Y_{i}(A)$ is binomially distributed for all Borel subsets $A$ of $E$. On the other hand, if $L_{i}-1$ is Poisson distributed, then we obtain \emph{Poisson} point processes. Although this assumption may be relaxed, we recall that throughout the paper $L_{i}$ takes values on $\mathbb{N}$ so that $L_{i}-1$ and not $L_{i}$ is Poisson distributed. It is also worth noticing that binomial and Poisson point processes are characterized by their intensity measure. Another important special case is that of \emph{Cox} processes. In this case, $L_{i}$ has the mixed Poisson distribution
\begin{equation*}
  \bP(L_{i}=k) = \int_{0}^{\infty} \frac{t^{k}}{k! (e^{t}-1)} \, d\nu(t),
\end{equation*}
for some probability distribution $\nu$ on $(0,\infty)$ (see (7.9) of \citet{Karr-1991}). The constant $-1$ in the denominator is due to the constraint on $L_{i}$. The following result for \emph{mixed binomial} point processes is an immediate consequence of Theorems \ref{theorem:uniform_LLN}-\ref{theorem:uniform_CLT} and the fact that $\mu(f)=\bE[L] \cdot \bE[f(X)]$, where $L$ and $X$ are independent copies of $L_{i}$ and $X_{i,j}$, respectively.

\begin{proposition} \label{proposition:mixed_binomial_point_processes}
Assume \ref{H1}-\ref{H2} and \ref{H5}. The following holds: \\
(i) If (i) or (ii) of Theorem \ref{theorem:uniform_LLN} holds, then
\begin{equation*}
  \sup_{f \in \mathcal{F}} \abs{\mu_{n}(f) - \bE[L] \cdot \bE[f(X)]} \xrightarrow[]{a.s.} 0.
\end{equation*}
(ii) If the classes of functions $\mathcal{F}^{\prime}_{\infty,2}$ and $\{ \mathcal{F}_{\delta,1}^{\prime} \}_{\delta>0}$ are measurable and \eqref{assumption_uniform_CLT} holds, then
\begin{equation*}
   \{ \sqrt{n} (\mu_{n}(f) - \bE[L] \cdot \bE[f(X)]) \}_{f \in \mathcal{F}} \xrightarrow[]{d} \{W(f)\}_{f \in \mathcal{F}},
\end{equation*}
where $W$ has covariance
\begin{equation*}
  \gamma(f,g) = \bE[L] \cdot \bE[ (f(X) - \bE[f(X)]) (g(X) - \bE[g(X)]) ].
\end{equation*}
\end{proposition}
We notice that, in the case of \emph{Cox} processes, Proposition \ref{proposition:mixed_binomial_point_processes} extends Theorems 7.4 and 7.5 of \citet{Karr-1991} to functions that are not necessarily continuous.

\section{Preliminary results} \label{section:preliminary_results}

In this section we derive a series of lemmas which are used in the proof of the main results of the paper. We begin with a precise definition of the underlying probability space and introduce some additional notation. We assume that the point processes $\{ Y_{i} \}_{i=1}^{\infty}$ are defined on a probability space $(\Omega_{Y}, \Sigma_{Y}, \bP_{Y})$. We will have occasion to consider independent copies of these point processes, which we assume to be defined on a probability space $(\Omega_{Y^{\prime}}, \Sigma_{Y^{\prime}}, \bP_{Y^{\prime}})$. We also need independent (of each other and of $\Sigma_{Y}$ and $\Sigma_{Y^{\prime}}$) Rademacher random variables $\{ \xi_{i} \}_{i=1}^{\infty}$ which we assume to be defined on the probability space $(\Omega_{\xi}, \Sigma_{\xi}, \bP_{\xi})$. Thus, the underlying probability space is the product space
\begin{equation*}
(\Omega, \Sigma, \bP) = (\Omega_{Y}, \Sigma_{Y}, \bP_{Y}) \times (\Omega_{Y^{\prime}}, \Sigma_{Y^{\prime}}, \bP_{Y^{\prime}}) \times (\Omega_{\xi}, \Sigma_{\xi}, \bP_{\xi}).
\end{equation*}
Expectations w.r.t.\ $\bP_{Y}$, $\bP_{Y^{\prime}}$, and $\bP_{\xi}$ are denoted by $\bE_{Y}$, $\bE_{Y^{\prime}}$, and $\bE_{\xi}$, respectively. Fubini Theorem for outer expectation (see Proposition 3.7.3 of \citet{Gine-2016}) yields that for any function $h:\Omega \to \mathbb{R}$
\begin{equation*}
  \bE_{Y}^{\ast} \bE_{\xi}^{\ast}[h] \leq \bE^{\ast}[h] \text{ and } \bE_{\xi}^{\ast} \bE_{Y}^{\ast}[h] \leq \bE^{\ast}[h],
\end{equation*}
where the superscript $^{\ast}$ is used to denote outer expectation. It holds that $\bE^{\ast}[h]=\bE[h^{\ast}]$ for some measurable function $h^{\ast}:\Omega \to \mathbb{R}$, which is referred to as the $\bP$-measurable cover of $h$. See Section 3.7.1 of \citet{Gine-2016} for more details. Similarly, we use the notation $\bP^{\ast}$ for outer probability.

We now study the empirical process $\mu_{n}$. It is clear that $\bE[\mu_{n}(f)]=\mu(f)$ whenever $\mu(f)$ exists and is finite and
\begin{equation} \label{variance}
  \bV[\mu_{n}(f)] = \frac{\sigma^{2}(f)}{n},
\end{equation}
whenever $\sigma^{2}(f) \coloneqq \bV[Y_{1}(f)] < \infty$. Similarly, we see that the covariance of $\mu_{n}$ is given by
\begin{equation} \label{covariance}
  \bC[\mu_{n}(f), \mu_{n}(g)] = \frac{\gamma(f,g)}{n},
\end{equation}
where $f,g \in \mathcal{F}$. The next lemma is inspired by Corollary 4.2 of \citet{Hoffmann-1974}. We define the symmetrization $\mu_{\xi,n}$ of $\mu_{n}$ by the Rademacher random variables $\{ \xi_{i} \}_{i=1}^{\infty}$ as follows:
\begin{equation*}
  \mu_{\xi,n}(f) \coloneqq \frac{1}{n} \sum_{i=1}^{n} \xi_{i} Y_{i}(f),
\end{equation*}
where $f \in \mathcal{F}$. We also let $\underline{\xi}_{n} \coloneqq \frac{1}{n} \sum_{i=1}^{n} \xi_{i}$.

\begin{lemma} \label{lemma:expectation_inequality}
  Assume \ref{H2} and $\bE[L_{1}^{p}] < \infty$ for $p \geq 1$. Then
\begin{equation*}
  2^{-p} \; \bE[\norm{\mu_{\xi,n} - \underline{\xi}_{n} \mu}^{p}] \leq \bE[\norm{\mu_{n}-\mu}^{p}] \leq 2^{p} \; \bE[\norm{\mu_{\xi,n}}^{p}].
\end{equation*}   
\end{lemma}

\begin{proof}[Proof of Lemma \ref{lemma:expectation_inequality}]
  We first notice that $\norm{\mu_{\xi,n} - \underline{\xi}_{n} \mu}$, $\norm{\mu_{n}-\mu}$, and $\norm{\mu_{\xi,n}}$ are measurable by \ref{H2}. We define $\bar{Y}_{i}(f) \coloneqq Y_{i}(f) - \mu(f)$, $C_{i}^{+} \coloneqq \delta_{1}(\xi_{i})$, and $C_{i}^{-} \coloneqq \delta_{-1}(\xi_{i})$. For the first inequality, we need to show that 
\begin{equation} \label{symmetrization_1}
  \bE_{Y}[\norm{\sum_{i=1}^{n} \xi_{i} \bar{Y}_{i} }^{p}] \leq 2^{p} \; \bE[\norm{\sum_{i=1}^{n} \bar{Y}_{i}}^{p}] \text{ a.s.,}
\end{equation}
where the expectation $\bE_{Y}$ is conditional on $\{ \xi_{i} \}_{i=1}^{\infty}$. Using Jensen's inequality, we obtain that
\begin{align*}
  \bE_{Y}[\norm{\sum_{i=1}^{n} \xi_{i} \bar{Y}_{i}}^{p}] = &\bE_{Y}[\norm{\sum_{i=1}^{n} C_{i}^{+} \, \bar{Y}_{i} - \sum_{i=1}^{n} C_{i}^{-} \, \bar{Y}_{i}}^{p}] \\
  \leq &2^{p-1} \biggl( \bE_{Y}[\norm{\sum_{i=1}^{n} C_{i}^{+} \, \bar{Y}_{i}}^{p}] + \bE_{Y}[\norm{\sum_{i=1}^{n} C_{i}^{-} \, \bar{Y}_{i}}^{p}] \biggr).
\end{align*}
Now, \eqref{symmetrization_1} follows if we show that
\begin{equation} \label{symmetrization_for_1}
  \bE_{Y}[\norm{\sum_{i=1}^{n} C_{i}^{\pm} \, \bar{Y}_{i}}^{p}] \leq \bE[\norm{\sum_{i=1}^{n} \bar{Y}_{i}}^{p}] \text{ a.s..}
\end{equation}
Let $\Sigma_{l}$ be the $\sigma$-algebra generated by the point processes $\{ Y_{i} \}_{i=1}^{l}$. Using again Jensen's inequality, we obtain that the LHS of \eqref{symmetrization_for_1} is bounded above by
\begin{equation*}
  (1-C_{n}^{\pm}) \; \bE_{Y}[\norm{\sum_{i=1}^{n-1} C_{i}^{\pm} \, \bar{Y}_{i}}^{p}] + C_{n}^{\pm} \; \bE_{Y}[\norm{\sum_{i=1}^{n-1} C_{i}^{\pm} \, \bar{Y}_{i} + \bar{Y}_{n} }^{p}].
\end{equation*}
Since $\bar{Y}_{n}$ has mean zero, we deduce that
\begin{align*}
  \bE_{Y}[\norm{\sum_{i=1}^{n-1} C_{i}^{\pm} \, \bar{Y}_{i}}^{p}] &= \bE_{Y}[\norm{ \bE_{Y}[\sum_{i=1}^{n-1} C_{i}^{\pm} \, \bar{Y}_{i} + \bar{Y}_{n} \, | \, \Sigma_{n-1}] }^{p}] \\
  &\leq \bE_{Y}[\norm{\sum_{i=1}^{n-1} C_{i}^{\pm} \, \bar{Y}_{i} + \bar{Y}_{n} }^{p}].
\end{align*}
It follows that
\begin{equation*}
  \bE_{Y}[\norm{\sum_{i=1}^{n} C_{i}^{\pm} \, \bar{Y}_{i}}^{p}] \leq \bE_{Y}[\norm{\sum_{i=1}^{n-1} C_{i}^{\pm} \, \bar{Y}_{i} + \bar{Y}_{n} }^{p}] \text{ a.s..}
\end{equation*}
By iterating the above argument, involving conditioning on $\Sigma_{j}$ and using that the function $\norm{\, \cdot \, + \sum_{i=j+1}^{n} \bar{Y}_{i}}^{p}$ is convex for $j=n-2,\dots,1$, we obtain \eqref{symmetrization_for_1}. For the second inequality, let
\begin{equation*}
  \mu_{n}^{\prime}(f) \coloneqq \frac{1}{n} \sum_{i=1}^{n} Y_{i}^{\prime}(f),
\end{equation*}
where $Y_{i}^{\prime}(f) \coloneqq \sum_{j=1}^{L_{i}^{\prime}} f(X_{i,j}^{\prime})$, be a copy of $\mu_{n}(f)$ defined on $(\Omega_{Y^{\prime}}, \Sigma_{Y^{\prime}}, \bP_{Y^{\prime}})$. Using that $\mu_{n}^{\prime}(f)$ has expectation $\mu(f)$ and Jensen's inequality, we obtain
\begin{align*}
  \bE[\norm{\mu_{n}-\mu}^{p}] &= \bE[\norm{\mu_{n}-\mu-E_{Y^{\prime}}[\mu_{n}^{\prime}-\mu]}^{p}] \\
  &= \bE[\norm{\bE_{Y^{\prime}}[\mu_{n}-\mu_{n}^{\prime}]}^{p}] \\
  &\leq \bE[\norm{\mu_{n}-\mu_{n}^{\prime}}^{p}].
\end{align*}
Since $Y_{i}-Y_{i}^{\prime}$ are symmetric, we conclude that
\begin{equation*}
  \bE[\norm{\mu_{n}-\mu_{n}^{\prime}}^{p}] = \bE[\norm{\mu_{\xi,n} - \mu_{\xi,n}^{\prime}}^{p}] \leq 2^{p} \; \bE[\norm{\mu_{\xi,n}}^{p}],
\end{equation*}
where $\mu_{\xi,n}^{\prime} \coloneqq \frac{1}{n} \sum_{i=1}^{n} \xi_{i} Y_{i}^{\prime}$.
\end{proof}

\noindent By Lemma \ref{lemma:expectation_inequality}, we can randomize the empirical process $\mu_{n}-\mu$ by Rademacher multipliers. The next lemma entails in particular that convergence in probability of the randomized process is equivalent to convergence of the original process (see also \citet{Gine-1984}). We let $\bar{\mu}_{\xi,n} \coloneqq \frac{1}{n} \sum_{i=1}^{n} \xi_{i} \bar{Y}_{i}$ and $\sigma^{2} \coloneqq \sup_{f \in \mathcal{F}} \sigma^{2}(f)$, where as before $\bar{Y}_{i}=Y_{i}-\mu$ and $\sigma^{2}(f) = \bV[Y_{1}(f)]$.

\begin{lemma} \label{lemma:probability_inequality}
  Assume \ref{H1}-\ref{H2}. Then, for $\epsilon>0$,
\begin{equation*}
  \bP(\norm{\bar{\mu}_{\xi,n}} \geq \epsilon) \leq 6 \cdot \bP(\norm{\mu_{n}-\mu} \geq \epsilon/20)
\end{equation*}  
and for $n \geq 8 \sigma^{2}/\epsilon^{2}$
\begin{equation*}
  \bP(\norm{\mu_{n}-\mu} \geq \epsilon) \leq 4 \cdot \bP(\norm{\mu_{\xi,n}} \geq \epsilon/4).
\end{equation*}
\end{lemma}
\begin{proof}[Proof of Lemma \ref{lemma:probability_inequality}]
For all $t \geq 0$, we have that
\begin{equation*}
\bP(\norm{\sum_{i=1}^{n} \xi_{i} \bar{Y}_{i}} \geq t) = 2^{-n}\cdot \hspace{-0.7cm} \sum_{(a_{1}, \dots, a_{n}) \in \{-1,1\}^{n}} \hspace{-0.2cm} \bP(\norm{\sum_{i=1}^{n} a_{i} \bar{Y}_{i}} \geq t),
\end{equation*}
where
\begin{align*}
\bP(\norm{\sum_{i=1}^{n} a_{i} \bar{Y}_{i}} \geq t) &\leq \bP(\norm{\sum_{i \in \{1,\dots,n\} : a_{i}=-1} \bar{Y}_{i}} \geq t/2) + \bP(\norm{\sum_{i \in \{1,\dots,n\} : a_{i}=1} \bar{Y}_{i}} \geq t/2) \\
&\leq 2 \max_{k=1,\dots,n} \bP(\norm{\sum_{i=1}^{k} \bar{Y}_{i}} \geq t/2).
\end{align*}
Theorem 1 of \citet{Montgomery-Smith-1993}, which remains valid for random variables in $(\ell_{\infty}(\mathcal{F}), C_{\infty}(\mathcal{F}))$ where $\mathcal{F}$ is measurable, yields that
\begin{equation*}
  \max_{k=1,\dots,n} \bP(\norm{\sum_{i=1}^{k} \bar{Y}_{i}} \geq t/2) \leq 3 \cdot \bP(\norm{\sum_{i=1}^{n} \bar{Y}_{i}} \geq t/20).
\end{equation*}
We conclude that 
\begin{equation*}
  \bP(\norm{\sum_{i=1}^{n} \xi_{i} \bar{Y}_{i}} \geq t) \leq 6 \cdot \bP(\norm{\sum_{i=1}^{n} \bar{Y}_{i}} \geq t/20)
\end{equation*}
and the first inequality follows by taking $t=n\epsilon$. Turning to the second inequality, we notice that by \eqref{variance}
\begin{equation*}  
  \sup_{f \in \mathcal{F}} \bE[(\mu_{n}-\mu)(f))^{2}] = \frac{\sigma^{2}}{n}.
\end{equation*}
  Using Proposition 3.1.24 (b) of \citet{Gine-2016} with $T=\mathcal{F}$, $t=f$, $Y(f) = (\mu_{n}-\mu)(f)$, $Y^{\prime}(f) = (\mu_{n}^{\prime}-\mu)(f)$, and $\theta= \sigma^{2}/n$, we obtain that for $s \geq \sqrt{2 \sigma^{2}/n}$
\begin{equation*}
  \bP(\norm{\mu_{n}-\mu} \geq s) \leq 2 \cdot \bP(\norm{\mu_{n}-\mu_{n}^{\prime}} \geq s -  \sqrt{2 \sigma^{2}/n} ).
\end{equation*}
Since $Y_{i}-Y_{i}^{\prime}$ are symmetric, we have that for all $t \geq 0$
\begin{equation*}
  \bP(\norm{\mu_{n}-\mu_{n}^{\prime}} \geq t ) \leq \bP(\norm{\mu_{\xi,n}-\mu_{\xi,n}^{\prime}} \geq t ).
\end{equation*}
Using this with $t = s - \sqrt{2 \sigma^{2}/n}$, we obtain
\begin{align*}
  \bP(\norm{\mu_{n}-\mu} \geq s) &\leq 2 \cdot \bP(\norm{\mu_{\xi,n}-\mu_{\xi,n}^{\prime}} \geq s -  \sqrt{2 \sigma^{2}/n} ) \\
  &\leq 4 \cdot \bP(\norm{\mu_{\xi,n}} \geq (s-\sqrt{2 \sigma^{2}/n})/2).
\end{align*}
The second inequality follows by taking $s=\epsilon$ and noticing that $\sqrt{2 \sigma^{2}/n} \leq \epsilon/2$ for all $n \geq 8 \sigma^{2}/\epsilon^{2}$.
\end{proof}

\noindent The next lemma shows that the process $\{ \sqrt{n} \mu_{\xi,n}(f) \}_{f \in \mathcal{F}}$ is subgaussian on $(\mathcal{F}, e_{n,2})$. We recall that a square integrable stochastic process $\{ X(t) \}_{t \in T}$ on a pseudo-metric space $(T, e)$ is subgaussian if for all $\lambda \in \mathbb{R}$ and $s,t \in T$
\begin{equation*}
  \bE[\exp(\lambda (X(s)-X(t)))] \leq \exp (\lambda^{2} e^{2}(s,t)/2).
\end{equation*}
\begin{lemma} \label{lemma:subgaussian}
  Assume \ref{H2}. Then, for all $\lambda \in \mathbb{R}$
\begin{equation*}
  \bE_{\xi}[\exp(\lambda \sqrt{n} (\mu_{\xi,n}(f)-\mu_{\xi,n}(g)))] \leq \exp(\lambda^{2} e_{n,2}^{2}(f,g)/2)
\end{equation*}
and, for all $f_{1},\dots,f_{N} \in \mathcal{F}$,
\begin{equation*}  
\bE_{\xi}[ \max_{j=1,\dots,N} \abs{\mu_{\xi,n}(f_{j})} ] \leq \sqrt{2 \log(2N)/n} \cdot \max_{j=1,\dots,N} e_{n,2}(f_{j},0).
\end{equation*}
\end{lemma}
\begin{proof}[Proof of Lemma \ref{lemma:subgaussian}]
Using that
\begin{align*}
  &\bE_{\xi}[\exp(a \xi_{i})] = (\exp(-a)+\exp(a))/2, \\
  &\prod \exp(-a_{j}) + \prod \exp(a_{j}) \leq \prod (\exp(-a_{j}) + \exp(a_{j})), \text{ and } \\
  &(\exp(-a)+\exp(a))/2 \leq \exp(a^{2}/2),
\end{align*}
we obtain
\begin{align*}
  &\bE_{\xi}[\exp(\lambda \sqrt{n} (\mu_{\xi,n}(f)-\mu_{\xi,n}(g)))] \\
  =& \prod_{i=1}^{n} ((\exp(-(\lambda/\sqrt{n}) Y_{i}(f-g))+\exp((\lambda/\sqrt{n}) Y_{i}(f-g)))/2) \\
  \leq& \prod_{i=1}^{n} \prod_{j=1}^{L_{i}} ((\exp(-(\lambda/\sqrt{n}) (f-g)(X_{i,j}))+\exp((\lambda/\sqrt{n}) (f-g)(X_{i,j})))/2) \\
  \leq& \prod_{i=1}^{n} \prod_{j=1}^{L_{i}} \exp(\lambda^{2} (f-g)^{2}(X_{i,j})/(2n)) =\exp(\lambda^{2} e_{n,2}^{2}(f,g)/2) \text{ a.s..}
\end{align*}
By taking $g=0$ and replacing $\lambda$ with $\lambda/\sqrt{n}$, we see that
\begin{equation*}
  \bE_{\xi}[\exp(\lambda \mu_{\xi,n}(f))] \leq \exp (\lambda^{2} e_{n,2}^{2}(f,0)/(2n) ) \text{ a.s.\ for all } f \in \mathcal{F},
\end{equation*}
that is, the random variables $\mu_{\xi,n}(f)$ are subgaussian with constant $e_{n,2}(f,0)/\sqrt{n}$. To conclude, we apply Lemma 2.3.4 of \citet{Gine-2016} with $\xi_{j} = \mu_{\xi,n}(f_{j})$ and $\sigma_{j} = e_{n,2}(f_{j},0)/\sqrt{n}$.
\end{proof}

\noindent We conclude this section with an inequality for $L^{p}$ empirical pseudo-distances. For $p=\infty$ we let
\begin{equation*}
  e_{n,\infty}(f,g) \coloneqq \max_{i=1,\dots,n} \max_{j=1,\dots,L_{i}} \abs{f-g}(X_{i,j}), \text{ where } f,g \in \mathcal{F}.
\end{equation*}

\begin{lemma} \label{lemma:inequality_Lp_empirical_distances}
  For all $1 \leq p \leq q \leq \infty$, it holds that
\begin{equation*}
  e_{n,p}(f,g) \leq \biggl( \frac{S_{n}}{n} \biggr)^{\frac{1}{p} - \frac{1}{q}} e_{n,q}(f,g), \text{ where } f,g \in \mathcal{F}.
\end{equation*}    
\end{lemma}

\begin{proof}[Proof of Lemma \ref{lemma:inequality_Lp_empirical_distances}]
For $f,g \in \mathcal{F}$ let
\begin{equation*}
  \tilde{e}_{n,p}(f,g) \coloneqq \begin{cases}
    \biggl[ \frac{1}{S_{n}} \sum_{i=1}^{n} \sum_{j=1}^{L_{i}} \abs{f-g}^{p}(X_{i,j}) \biggr]^{\frac{1}{p}} &\text{ for } p < \infty, \\
    e_{n,\infty}(f,g) &\text{ for } p = \infty.
\end{cases}    
\end{equation*}  
The claim follows from $e_{n,p}(f,g) = \left( \frac{S_{n}}{n} \right)^{\frac{1}{p}} \tilde{e}_{n,p}(f,g)$ and $\tilde{e}_{n,p}(f,g) \leq \tilde{e}_{n,q}(f,g)$.
\end{proof}

\section{Proofs of main results} \label{section:proofs_of_main_results}

In this section we prove the main results of the paper.

\begin{proof}[Proof of Theorem \ref{theorem:uniform_LLN}]
We prove that (i) implies that $\norm{\mu_{n}-\mu} \xrightarrow[]{L^{1}} 0$. Using Lemma \ref{lemma:expectation_inequality} with $p=1$, we obtain
\begin{equation} \label{symmetrization}
  \bE[\norm{\mu_{n}-\mu}] \leq 2 \; \bE[\norm{\mu_{\xi,n}}].
\end{equation}
It is enough to show that $\norm{\mu_{\xi,n}} \xrightarrow[]{L^{1}} 0$. For $\epsilon>0$, let $f_{1},\dots,f_{N} \in \mathcal{F}$ such that $\min_{i=1,\dots,N} e_{n,p}(f,f_{i}) \leq \epsilon$ for all $f \in \mathcal{F}$, where $N=N(\mathcal{F},e_{n,p},\epsilon)$ is the covering number of the pseudo-metric space $(\mathcal{F},e_{n,p})$. Then, for each $f \in \mathcal{F}$, there is $j(f)$ such that $e_{n,p}(f,f_{j(f)}) \leq \epsilon$. By adding and subtracting $f_{j(f)}$, we see that
\begin{equation*}
  \bE_{\xi}[\norm{\mu_{\xi,n}}] \leq \bE_{\xi}[\sup_{f \in \mathcal{F}}\abs{\mu_{\xi,n}(f-f_{j(f)})}] + \bE_{\xi}[\max_{j=1,\dots,N} \abs{\mu_{\xi,n}(f_{j})}].
\end{equation*}
For the first term, using Lemma \ref{lemma:inequality_Lp_empirical_distances}, we have
\begin{align*}
  \bE_{\xi}[\sup_{f \in \mathcal{F}}\abs{\mu_{\xi,n}(f-f_{j(f)})}] &\leq \sup_{f \in \mathcal{F}} e_{n,1}(f,f_{j(f)}) \\
  &\leq \biggl( \frac{S_{n}}{n} \biggr)^{1-1/p} \sup_{f \in \mathcal{F}} e_{n,p}(f,f_{j(f)}) \leq \biggl( \frac{S_{n}}{n} \biggr)^{1-1/p} \epsilon.
\end{align*}  
Using Lemma \ref{lemma:subgaussian} and $f_{j}^{2} \leq M^{2}$ in the second term, where $M$ is the bound given by \ref{H2}, we obtain
\begin{equation*}
  \bE_{\xi}[\max_{j=1,\dots,N} \abs{\mu_{\xi,n}(f_{j})}] \leq M \biggl( \frac{2 \log(2N)}{n} \cdot \frac{S_{n}}{n} \biggr)^{1/2}.
\end{equation*}
We deduce that $\bE_{\xi}[\norm{\mu_{\xi,n}}] \leq R_{n,p,\epsilon}$, where
\begin{equation} \label{bound_symmetrization}
  R_{n,p,\epsilon} \coloneqq \biggl( \frac{S_{n}}{n} \biggr)^{1-1/p}\epsilon + M \biggl( \frac{2 \log(2N)}{n} \cdot \frac{S_{n}}{n} \biggr)^{1/2},
\end{equation}
and by taking the expectation $\bE_{Y}$ we obtain
\begin{equation} \label{dominating_variable}
  \bE[\norm{\mu_{\xi,n}}] \leq \bE[R_{n,p,\epsilon}].
\end{equation}  
We show below that the sequence $\{ R_{n,p,\epsilon} \}_{n=1}^{\infty}$ is uniformly integrable. Using condition (i), $S_{n}/n \xrightarrow[]{a.s.} \bE[L_{1}]$, and Vitali convergence theorem, we obtain that
\begin{equation*}
    \limsup_{n \to \infty} \bE[\norm{\mu_{\xi,n}}] \leq \bE[L_{1}]^{1-1/p} \epsilon
\end{equation*}
yielding the desired $L^{1}$ convergence. We now show uniform integrability of the sequence $\{ R_{n,p,\epsilon} \}_{n=1}^{\infty}$. We notice that for all $f \in \mathcal{F}$ each random variable $f(X_{i,j})$ $j=1,\dots,L_{i}$ belongs to the interval $[-M,M]$. Since $[-M,M]^{S_{n}}$ can be covered by at most $(1+M/l)^{S_{n}}$ hypercubes with side length $2 l$, using Lemma \ref{lemma:inequality_Lp_empirical_distances} and $S_{n} \geq n$ a.s., we obtain that for $\epsilon \leq M$
\begin{equation} \label{upper_bound}
\begin{aligned}  
  N = N(\mathcal{F}, e_{n,p}, \epsilon) &\leq N(\mathcal{F}, \biggl( \frac{S_{n}}{n} \biggr)^{1/p} e_{n,\infty}, \epsilon) \\
  &\leq \biggl( 1+\frac{M}{\epsilon} \cdot \biggl( \frac{S_{n}}{n} \biggr)^{1/p} \biggr)^{S_{n}} \leq \biggl( \frac{2M}{\epsilon} \cdot \biggl( \frac{S_{n}}{n} \biggr)^{1/p} \biggr)^{S_{n}}.
\end{aligned}  
\end{equation}
Therefore,
\begin{equation*}
  \frac{\log(2N)}{n} \leq \frac{S_{n}}{n} \cdot \log \biggl( \frac{4M}{\epsilon} \cdot \biggl( \frac{S_{n}}{n} \biggr)^{1/p}\biggr) \leq \frac{4M}{\epsilon} \cdot \biggl( \frac{S_{n}}{n} \biggr)^{1+1/p}.
\end{equation*}
Using this in \eqref{bound_symmetrization}, we obtain that $R_{n,p,\epsilon} \leq T_{n,p,\epsilon}$, where
\begin{equation*}
  T_{n,p,\epsilon} \coloneqq \biggl( \frac{S_{n}}{n} \biggr)^{1-1/p} \epsilon + \biggl( \frac{S_{n}}{n} \biggr)^{1+1/(2p)} \biggl(\frac{8M^{3}}{\epsilon} \biggr)^{1/2}.
\end{equation*}
Using \ref{H1}, we conclude that $\{ T_{n,p,\epsilon} \}_{n=1}^{\infty}$ is uniformly integrable and so it is $\{R_{n,p,\epsilon}\}_{n=1}^{\infty}$. We now show that condition (ii) implies $\norm{\mu_{n} - \mu } \xrightarrow[]{p} 0$. Using Lemma \ref{lemma:probability_inequality} and Markov's inequality, we obtain that for all $\epsilon>0$ and $n \geq 8 \sigma^{2}/\epsilon^{2}$
\begin{equation*}
  \bP(\norm{\mu_{n}-\mu} \geq \epsilon) \leq 4 \cdot \bE_{Y}[\min(1,\frac{4}{\epsilon} \bE_{\xi}[\norm{\mu_{\xi,n}}])]
\end{equation*}
Next, Lemma \ref{lemma:subgaussian} and Theorem 2.3.7 of \citet{Gine-2016} with $T=\mathcal{F} \cup \{ 0 \}$, $d=e_{n,2}$, $t=f$, $X(f)=\sqrt{n} \mu_{\xi,n}(f)$, $t_{0}=0$, $\bE=\bE_{\xi}$, and $D=2M\sqrt{S_{n}/n}$ yields that
\begin{equation*}
  \bE_{\xi}[\norm{\mu_{\xi,n}}] \leq \frac{4 \sqrt{2}}{\sqrt{n}} \int_{0}^{M\sqrt{\hspace{-0.05cm} \frac{S_{n}}{n}}} \hspace{-0.4cm} \sqrt{\log(2 N(\mathcal{F}, e_{n,2}, \epsilon))} \, d \epsilon.
\end{equation*}
We deduce that
\begin{equation*}
    \bP(\norm{\mu_{n}-\mu} \geq \epsilon) \leq 4 \, \bE_{Y}[\min(1, \frac{16 \sqrt{2}}{\epsilon \sqrt{n}} \int_{0}^{M\sqrt{\hspace{-0.05cm} \frac{S_{n}}{n}}} \hspace{-0.4cm} \sqrt{\log(2N^{\ast}(\mathcal{F}, e_{n,2}, \epsilon))} \, d \epsilon)].
\end{equation*}
Using \ref{H1}, $S_{n}/n \xrightarrow[]{a.s.} \bE[L_{1}]$, and (ii) we obtain that this converges to zero as $n \to \infty$.

To conclude, we show that $\norm{\mu_{n} - \mu } \xrightarrow[]{p} 0$ implies $\norm{\mu_{n} - \mu } \xrightarrow[]{a.s.} 0$. We first notice that $\norm{\mu_{\xi,n}} \leq \norm{\bar{\mu}_{\xi,n}} + \underline{\xi}_{n} \norm{\mu}$ and $\underline{\xi}_{n} \xrightarrow[]{a.s.} 0$. By Lemma \ref{lemma:probability_inequality} and Borel-Cantelli lemma, it is enough to show that $\norm{\mu_{\xi,n}} \xrightarrow[]{p} 0$ implies $\norm{\mu_{\xi,n}} \xrightarrow[]{a.s.} 0$. Let $\underline{U}_{n} \coloneqq \frac{1}{n} \sum_{i=1}^{n} U_{i}$, where $U_{i} \coloneqq \xi_{i} Y_{i} \mathbf{I}_{[0,i]}(\norm{\xi_{i} Y_{i}})$. Since
\begin{equation} \label{finite_moment_for_almost_sure_uniform_LLN}
  \bE[\norm{\xi_{i} Y_{i}}^{2}] \leq M^{2} \cdot \bE[L_{1}^{2}] < \infty,
\end{equation}
we have that
\begin{equation*}
  \sum_{i=1}^{\infty} \bP( \norm{\xi_{i} Y_{i}} > i) < \infty
\end{equation*}
yielding $\norm{\mu_{\xi,n}-\underline{U}_{n}} \xrightarrow[]{a.s.} 0$. Thus, we are left to show that $\norm{\underline{U}_{n}} \xrightarrow[]{p} 0$ implies $\norm{\underline{U}_{n}} \xrightarrow[]{a.s.} 0$. This can be obtained using Lemma \ref{lemma:almost_sure_uniform_LLN} below. Specifically, using Lemma \ref{lemma:almost_sure_uniform_LLN} (ii) and independence we obtain
\begin{equation*}
  \lim_{n \to \infty} \bE[\norm{\sum_{i=2^{n}+1}^{2^{n+1}} U_{i}}] /2^{n} = 0.
\end{equation*}
By the above equality and Lemma \ref{lemma:almost_sure_uniform_LLN} (iii), it suffices to show that for all $\epsilon >0$ $\sum_{n=1}^{\infty} \bP(\abs{J_{n}} \geq \epsilon) < \infty$, where
\begin{equation*}
  J_{n} \coloneqq \biggl( \norm{\sum_{i=2^{n}+1}^{2^{n+1}} U_{i}} - \bE[\norm{\sum_{i=2^{n}+1}^{2^{n+1}} U_{i}}] \biggr)/2^{n}.
\end{equation*}
Using Lemma \ref{lemma:almost_sure_uniform_LLN} (i), $\bE[\norm{U_{i}}^{2}] \leq \bE[\norm{\xi_{i} Y_{i}}^{2}]$, and \eqref{finite_moment_for_almost_sure_uniform_LLN}, we conclude that
\begin{equation*}
  \bP(\abs{J_{n}} \geq \epsilon) \leq \frac{\bE[J_{n}^{2}]}{\epsilon^{2}} \leq \frac{4}{\epsilon^{2}} \cdot \frac{1}{2^{2n}} \cdot \sum_{i=2^{n}+1}^{2^{n+1}} \bE[\norm{U_{i}}^{2}] \leq \frac{4 \cdot M^{2} \cdot \bE[L_{1}^{2}]}{\epsilon^{2}} \cdot \frac{1}{2^{n}}.
\end{equation*}
\end{proof}

\noindent The following lemma is based on Theorem 2.1 of \citet{deAcosta-1981} and Lemmas 2.2-2.3 of \citet{Kuelbs-1979}, which are for Banach spaced-valued random variables with the Borel $\sigma$-algebra. The difference with the next lemma is that the $\sigma$-algebra associated with $\ell_{\infty}(\mathcal{F})$ is that the $\sigma$-algebra $C_{\infty}(\mathcal{F})$ is not the Borel $\sigma$-field. The random variables $\{ U_{i} \}_{i=1}^{n}$ are defined in the proof of Theorem \ref{theorem:uniform_LLN} and they are independent and symmetric, which is used in the proof of (ii) and (iii) below.

\begin{lemma} \label{lemma:almost_sure_uniform_LLN}
  Assume \ref{H1}-\ref{H2}. The following holds: \\
  (i) For all $1 \leq k < n$
\begin{equation*}
  \bE[ \biggl( \norm{\sum_{i=k}^{n} U_{i}} - \bE[\norm{\sum_{i=k}^{n} U_{i}}] \biggr)^{2}] \leq 4 \sum_{i=k}^{n} \bE[\norm{U_{i}}^{2}],
\end{equation*}
(ii) $\norm{U_{i}} \leq i$ and $\norm{\underline{U}_{n}} \xrightarrow[]{p} 0$ imply that $\lim_{n \to \infty} \bE[\norm{\underline{U}_{n}}] = 0$, and \\
(iii)  $\norm{\underline{U}_{n}} \xrightarrow[]{a.s.} 0$ if and only if $\norm{\sum_{i=2^{n}+1}^{2^{n+1}} U_{i} /2^{n} } \xrightarrow[]{a.s.} 0$.
\end{lemma}

\begin{proof}[Proof of Lemma \ref{lemma:almost_sure_uniform_LLN}]
We begin by proving (i). Let $\mathcal{U}_{j}$ be the $\sigma$-algebra generated by $\{ U_{i} \}_{i=k}^{j}$. We notice that
\begin{equation*}
  \norm{\sum_{i=k}^{n} U_{i}} - \bE[\norm{\sum_{i=k}^{n} U_{i}} = \sum_{j=k}^{n} I_{j},
\end{equation*}
where
\begin{equation*}
  I_{j} \coloneqq \bE[ \norm{\sum_{i=k}^{n} U_{i}} \, | \, \mathcal{U}_{j}] - \bE[ \norm{\sum_{i=k}^{n} U_{i}} \, | \, \mathcal{U}_{j-1}].
\end{equation*}
Since for $j<l$
\begin{equation*}
  \bE[I_{j} I_{l} ] = \bE[ \bE[I_{l} \, | \, \mathcal{U}_{l-1}] \cdot I_{j}] = 0,
\end{equation*}
we obtain
\begin{equation*}
  \bE[ \biggl( \norm{\sum_{i=k}^{n} U_{i}} - \bE[\norm{\sum_{i=k}^{n} U_{i}}] \biggr)^{2}] = \sum_{j=k}^{n} \bE[ I_{j}^{2} ].
\end{equation*}
Using
\begin{equation*}
  \abs{ \norm{\sum_{i=k}^{n} U_{i}} - \norm{\sum_{\substack{i=k \\ i \neq j}}^{n} U_{i}} } \leq \norm{U_{j}}
\end{equation*}
and
\begin{equation*}
  \bE[ \norm{\sum_{\substack{i=k \\ i \neq j}}^{n} U_{i} } \, | \, \mathcal{U}_{j}] = \bE[ \norm{\sum_{\substack{i=k \\ i \neq j}}^{n} U_{i} } \, | \, \mathcal{U}_{j-1}],
\end{equation*}
we see that
\begin{equation*}
  \abs{I_{j}} \leq \norm{U_{j}} + \bE[\norm{U_{j}}],
\end{equation*}
which in turn gives $\bE[ I_{j}^{2} ] \leq 4 \cdot \bE[\norm{U_{j}}^{2}]$. We now turn to the proof of (ii). Let $\epsilon>0$ and $n^{*} \in \mathbb{N}$ such that $\bP(\norm{\underline{U}_{n}} \geq \epsilon/24) \leq 1/24$ for all $n \geq n^{*}$. We show that $\limsup_{n \to \infty} \bE[\norm{\underline{U}_{n}}] \leq \epsilon$. Lemma 3.4 of \citet{Jain-1975}, which remains valid for random variables in $(\ell_{\infty}(\mathcal{F}), C_{\infty}(\mathcal{F}))$ where $\mathcal{F}$ is measurable, yields that for all $t > 0$
\begin{equation*}
  \bP( \norm{\underline{U}_{n}} > 3t) \leq 4 \cdot \bP^{2}( \norm{\underline{U}_{n}} > t) + \bP( \max_{i=1,\dots,n} \norm{U_{i}}/n > t).
\end{equation*}
It follows that for all $n \geq n^{*}$
\begin{align*}
  \bE[ \norm{\underline{U}_{n}} ] &= 3 \int_{0}^{\infty} \bP( \norm{\underline{U}_{n}} > 3t) \, dt \\
  &\leq 12 \int_{0}^{\infty} \bP^{2}( \norm{\underline{U}_{n}} > t) \, dt + 3 \int_{0}^{\infty} \bP( \max_{i=1,\dots,n} \norm{U_{i}}/n > t) \, dt \\
  &\leq \epsilon/2 + \bE[ \norm{\underline{U}_{n}} ]/2 + 3 \cdot \bE[ \max_{i=1,\dots,n} \norm{U_{i}}/n ].
\end{align*}
We deduce that
\begin{equation*}
  \bE[ \norm{\underline{U}_{n}} ] \leq \epsilon + 6 \cdot \bE[ \max_{i=1,\dots,n} \norm{U_{i}}/n ].
\end{equation*}
Using that $\max_{i=1,\dots,n} \norm{U_{i}}/n \xrightarrow[]{p} 0$ and $\max_{i=1,\dots,n} \norm{U_{i}}/n \leq 1$, we conclude that
\begin{equation*}
  \lim_{n \to \infty} \bE[ \max_{i=1,\dots,n} \norm{U_{i}}/n ] = 0.
\end{equation*}  
We now prove (iii). If $\norm{\underline{U}_{n}} \xrightarrow[]{a.s.} 0$, then
\begin{equation*}
  \norm{\sum_{i=2^{n}+1}^{2^{n+1}} U_{i} /2^{n} } \leq 2 \norm{\underline{U}_{2^{n+1}}} + \norm{\underline{U}_{2^{n}}} \xrightarrow[]{a.s.} 0.
\end{equation*}
On the other hand, for $2^{k}+1 \leq n \leq 2^{k+1}$ we have that
\begin{equation*}
  \norm{\underline{U}_{n}} \leq \norm{ \sum_{i=2^{k}+1}^{n} U_{i}/2^{k} } + \norm{\underline{U}_{2^{k}}}.
\end{equation*}
Using Toeplitz lemma and $\norm{\sum_{i=2^{j-1}+1}^{2^{j}} U_{i} /2^{j-1} } \xrightarrow[]{a.s.} 0$, we obtain that
\begin{equation*}
  \norm{\underline{U}_{2^{k}}} \leq \sum_{j=1}^{k} \frac{2^{j-1}}{2^{k}} \norm{\sum_{i=2^{j-1}+1}^{2^{j}} U_{i} /2^{j-1} } + \frac{\norm{U_{1}}}{2^{k}} \xrightarrow[]{a.s.} 0.
\end{equation*}
Since $\{ U_{i} \}_{i=2^{k}+1}^{2^{k+1}}$ are symmetric, Levy's inequality (see e.g.\ Lemma 2.1 of \citet{Jain-1975} and Section 32.1.A of \citet{Loeve-1977}) yields that for all $\epsilon>0$
\begin{equation*}
  \bP( \max_{n=2^{k}+1, \dots, 2^{k+1}} \norm{ \sum_{i=2^{k}+1}^{n} U_{i}/2^{k} } \geq \epsilon ) \leq 2 \cdot \bP( \norm{ \sum_{i=2^{k}+1}^{2^{k+1}} U_{i}/2^{k} } \geq \epsilon).
\end{equation*}
The claim now follows from Borel-Cantelli lemma.
\end{proof}

\noindent Before proving Proposition \ref{proposition:VC-subgraph_bound} we recall that the packing numbers of the pseudo-metric space $(T, e)$ are given for any $\epsilon>0$ by
\begin{equation*}
  D(T, e, \epsilon) \coloneqq \sup \{N : \exists t_{1}, \dots, t_{N} \in T \text{ such that } \min_{i,j=1,\dots,N, \; i \neq j} e(t_{i},t_{j}) > \epsilon \}.
\end{equation*}
Covering and packing numbers are equivalent in the sense that for all $\epsilon>0$
\begin{equation} \label{covering_and_packing_numbers}
  N(T,e,\epsilon) \leq D(T,e,\epsilon) \leq N(T,e,\epsilon/2)
\end{equation}
whenever $\abs{T} \geq 2$.

\begin{proof}[Proof of Proposition \ref{proposition:VC-subgraph_bound}]
Since $\mathcal{F}$ is non-empty, $\mathbf{v} \geq 1$. Also, $\mathbf{v}=1$ if and only if $\mathcal{F}$ consists of a single function, in which case $N(\mathcal{F}, e_{n,p}, \epsilon)=1$ and \eqref{VC-subgraph_bound} holds with $c_{\mathbf{v}} \coloneqq 1$. We assume from now on that $\mathbf{v} \geq 2$. We apply Theorem 3.6.9 of \citet{Gine-2016} and obtain that for $\mathbf{w}>\mathbf{v}-1$ and $c_{\mathbf{v},\mathbf{w}} \coloneqq \max \{ c \in \mathbb{N} : \log(c) \geq c^{1/(\mathbf{v}-1)-1/\mathbf{w}} \}$
\begin{equation*}
  D(\mathcal{F}, \tilde{e}_{n,p}, \epsilon) \leq \max \biggl( c_{\mathbf{v},\mathbf{w}}, 2^{\frac{\mathbf{w} \mathbf{v}}{\mathbf{v}-1}} \cdot \biggl(\frac{2M}{\epsilon}\biggr)^{p \mathbf{w}} \biggr).
\end{equation*}
The claim now follows from \eqref{covering_and_packing_numbers} and $e_{n,p} = \left( \frac{S_{n}}{n} \right)^{\frac{1}{p}} \tilde{e}_{n,p}$, taking $\mathbf{w}=\mathbf{v}$ and letting $c_{\mathbf{v}} \coloneqq c_{\mathbf{v},\mathbf{v}}$.
\end{proof}

\noindent We now turn to the proof of Theorem \ref{theorem:uniform_CLT} and Corollary \ref{corollary:uniform_CLT}. We define the $L^{2}(\mu)$ distance $e_{\mu}$ on $\mathcal{F}$ by
\begin{equation*}
  e_{\mu}^{2}(f,g) \coloneqq \norm{f-g}_{L^{2}(\mu)}^{2} = \mu(\abs{f-g}^{2}) \text{ for } f,g \in \mathcal{F}.
\end{equation*}

\begin{proof}[Proof of Theorem \ref{theorem:uniform_CLT}]
  By Theorem 3.7.23 of \citet{Gine-2016} it is enough to show that (a) the finite dimensional distributions of the process $W_{n} \coloneqq \sqrt{n} (\mu_{n}-\mu)$ converge in law, (b) the space $(\mathcal{F},e_{\mu})$ is totally bounded, and (c) the process $W_{n}$ is asymptotically equicontinuous, that is, for all $\epsilon>0$
\begin{equation} \label{asymptotic_equicontinuity}
  \lim_{\delta \to 0^{+}} \limsup_{n \to \infty} \bP^{\ast}(\sup_{f,g \in \mathcal{F} \, : \, e_{\mu}(f,g) \leq \delta} \abs{W_{n}(f)-W_{n}(g)} \geq \epsilon) = 0.
\end{equation}
For (a) let $f_{1}, f_{2}, \dots, f_{k} \in \mathcal{F}$ and consider the vector
\begin{equation} \label{multivariate_CLT}
  \sqrt{n} (\mathbf{F}_{n}-\mathbf{F}) = \frac{1}{\sqrt{n}} \sum_{i=1}^{n} (\mathbf{Y}_{i}-\mathbf{F}),
\end{equation}
where
\begin{equation*}
  \mathbf{F}_{n} \coloneqq \begin{pmatrix}
    \mu_{n}(f_{1}) \\
    \mu_{n}(f_{2}) \\
    \vdots \\
    \mu_{n}(f_{k})
  \end{pmatrix},
\mathbf{Y}_{i} \coloneqq \begin{pmatrix}
    Y_{i}(f_{1}) \\
    Y_{i}(f_{2}) \\
    \vdots \\
    Y_{i}(f_{k}) \\
\end{pmatrix},  \text{ and }
  \mathbf{F} \coloneqq \begin{pmatrix}
    \mu(f_{1}) \\
    \mu(f_{2}) \\
    \vdots \\
    \mu(f_{k})
\end{pmatrix}.
\end{equation*}
As $\{ \mathbf{Y}_{i} \}_{i=1}^{n}$ are independent with mean $\bE[\mathbf{Y}_{1}]=\mathbf{F}$ and covariance matrix $\bE[\mathbf{Y}_{1} \mathbf{Y}_{1}^{\top}] = \mathbf{\Sigma}$, where $(\mathbf{\Sigma})_{i,j} = \gamma(f_{i},f_{j})$ (see also \eqref{covariance}), the multivariate CLT yields that $\sqrt{n} (\mathbf{F}_{n}-\mathbf{F})$ is asymptotically normal with mean $\mathbf{0}$ and covariance matrix $\mathbf{\Sigma}$. This concludes the proof of (a) and also gives the finite dimensional distributions of the limiting Gaussian process $W$. We now turn to the proof of (b). We first show that the class $\mathcal{F}_{\infty,2}^{\prime}$ is Glivenko-Cantelli. Using that
\begin{align*}
  \abs{(f_{1}-g_{1})^{2}-(f_{2}-g_{2})^{2}} &= \abs{(f_{1}-g_{1})+(f_{2}-g_{2})} \abs{(f_{1}-g_{1})-(f_{2}-g_{2})} \\
  &\leq 4 M(\abs{f_{1}-f_{2}}+\abs{g_{1}-g_{2}}),
\end{align*}  
we see that
\begin{equation*}
  N(\mathcal{F}_{\infty,2}^{\prime}, e_{n,2}, \epsilon) \leq N^{2}(\mathcal{F}, e_{n,2}, \epsilon/(8M)).
\end{equation*}
We deduce that
\begin{equation*}
  \frac{1}{\sqrt{n}} \int_{0}^{\delta} \sqrt{\log(N^{\ast}(\mathcal{F}_{\infty,2}^{\prime},e_{n,2},\epsilon))} \, d\epsilon \leq \frac{8M\sqrt{2}}{\sqrt{n}} \int_{0}^{\delta/(8M)} \sqrt{\log(N^{\ast}(\mathcal{F},e_{n,2},\epsilon))} \, d\epsilon.
\end{equation*}
Since \eqref{assumption_uniform_CLT} implies condition (ii) of Theorem \ref{theorem:uniform_LLN}, the above inequality and Theorem \ref{theorem:uniform_LLN} imply that the class $\mathcal{F}_{\infty,2}^{\prime}$ is Glivenko-Cantelli. It follows that $\lim_{n \to \infty} \bP(A_{n,\epsilon})=1$, where $A_{n,\epsilon} \coloneqq \{ \norm{\mu_{n}-\mu}_{\mathcal{F}_{\infty,2}^{\prime}} \leq \epsilon \}$ and $\epsilon>0$. In particular, $\bP(A_{n,\epsilon})>0$ for all large $n$. Using \eqref{assumption_uniform_CLT} we obtain that for all large $n$ there is a set $B_{n,\epsilon} \subset \Omega$ of probability one such that $N(\omega) \coloneqq N(\mathcal{F},e_{n,2},\sqrt{\epsilon})(\omega) < \infty$ for all $\omega \in B_{n,\epsilon}$. Therefore, there exist $f_{1}, f_{2}, \dots, f_{N(\omega)} \in \mathcal{F}$ such that $\min_{i=1,\dots,N(\omega)} e_{n,2}(f,f_{i}) \leq \sqrt{\epsilon}$ for all $f \in \mathcal{F}$. It follows that for large $n$ and $\omega \in A_{n,\epsilon} \cap B_{n,\epsilon}$
\begin{equation*}
  \min_{i=1,\dots,N(\omega)} e_{\mu}^{2}(f,f_{i}) \leq \norm{\mu_{n}-\mu}_{\mathcal{F}_{\infty,2}^{\prime}}(\omega) + \min_{i=1,\dots,N(\omega)} e_{n,2}^{2}(f,f_{i})(\omega) \leq 2 \epsilon,
\end{equation*}
which entails that the balls with centers $\{ f_{i} \}_{i=1}^{N(\omega)}$ and radii $\sqrt{2 \epsilon}$ cover the space $(\mathcal{F},e_{\mu})$. As $\epsilon$ is arbitrary, $(\mathcal{F}, e_{\mu} )$ is totally bounded. Finally, we prove the asymptotic equicontinuity condition \eqref{asymptotic_equicontinuity}. Using that
\begin{equation*}
  e_{n,2}^{2}(f,g) \leq e_{\mu}^{2}(f,g) + \delta^{2} \leq 2 \delta^{2}
\end{equation*}
for all $(f-g)^{2} \in \mathcal{F}_{\infty,2}^{\prime}$ such that $\abs{e_{n,2}^{2}(f,g) - e_{\mu}^{2}(f,g)} \leq \delta^{2}$, we obtain
\begin{align*}
  \bP( \sup_{f,g \in \mathcal{F} \, : \, e_{\mu}(f,g) \leq \delta} \abs{W_{n}(f) - W_{n}(g)} \geq \epsilon) &\leq \bP( \sup_{f,g \in \mathcal{F} \, : \, e_{n,2}(f,g) \leq \sqrt{2} \delta} \abs{W_{n}(f) - W_{n}(g)} \geq \epsilon) \\
  &+ \bP(\sup_{(f-g)^{2} \in \mathcal{F}_{\infty,2}^{\prime}} \abs{e_{n,2}^{2}(f,g) - e_{\mu}^{2}(f,g)} > \delta^{2}).
\end{align*}
The second term in the above inequality equals $\bP(\norm{\mu_{n} - \mu}_{\mathcal{F}_{\infty,2}^{\prime}} > \delta^{2})$ and converges to zero as $n \to \infty$ because the class $\mathcal{F}_{\infty,2}^{\prime}$ is Glivenko-Cantelli. \eqref{asymptotic_equicontinuity} follows if we show that for $0<\epsilon<1$
\begin{equation} \label{asymptotic_equicontinuity_2}
  \lim_{\delta \to 0^{+}} \limsup_{n \to \infty} \bP( \sup_{f,g \in \mathcal{F} \, : \, e_{n,2}(f,g) \leq \sqrt{2} \delta} \abs{W_{n}(f) - W_{n}(g)} \geq \epsilon) = 0.
\end{equation}
By applying Markov's inequality and proceeding as in the proof of Lemma \ref{lemma:expectation_inequality}, we obtain that the probability in \eqref{asymptotic_equicontinuity_2} is bounded above by
\begin{align*}
  &\bE\biggl[\min \biggl(1, \frac{1}{\epsilon} \sup_{f,g \in \mathcal{F} \, : \, e_{n,2}(f,g) \leq \sqrt{2} \delta} \abs{W_{n}(f) - W_{n}(g)} \biggr) \biggr] \\
  \leq & \bE\biggl[\min \biggl(1, \frac{2}{\epsilon} \bE_{\xi} \biggl[ \sup_{f,g \in \mathcal{F} \, : \, e_{n,2}(f,g) \leq \sqrt{2} \delta} \biggl| \frac{1}{\sqrt{n}} \sum_{i=1}^{n} \xi_{i} (Y_{i}(f)-Y_{i}(g)) \biggr| \biggr] \biggr) \biggr].
\end{align*}
Now, using Theorem 2.3.7 (b) of \citet{Gine-2016} with $T=\mathcal{F}$, $d=e_{n,2}$, $t=f$, and $X(f) = \frac{1}{\sqrt{n}} \sum_{i=1}^{n} \xi_{i} Y_{i}(f)$ (see also Lemma \ref{lemma:subgaussian}), we see that the inner expectation in the above term is bounded above by
\begin{equation*}
  (16 \sqrt{2} + 2) \int_{0}^{\sqrt{2}\delta} \sqrt{\log(2 N(\mathcal{F}, e_{n,2}, \tau))} \, d\tau. 
\end{equation*}
We conclude that the probability in \eqref{asymptotic_equicontinuity_2} is bounded above by
\begin{equation*}
  (16 \sqrt{2} + 2) \sqrt{\log(2)} \cdot \frac{2 \sqrt{2}\delta}{\epsilon} + \bE\biggl[\min \biggl(1, \frac{2(16 \sqrt{2} + 2)}{\epsilon} \int_{0}^{\sqrt{2}\delta} \sqrt{\log(N^{\ast}(\mathcal{F}, e_{n,2}, \tau))} \, d\tau \biggl) \biggr].
\end{equation*}
Taking $\lim_{\delta \to 0^{+}} \limsup_{n \to \infty}$ and using \eqref{assumption_uniform_CLT} we obtain \eqref{asymptotic_equicontinuity_2}.
\end{proof}

\begin{proof}[Proof of Corollary \ref{corollary:uniform_CLT}]
  It suffices to show that \ref{H3} implies \eqref{assumption_uniform_CLT}. To this end, let $\log_{+}(t) \coloneqq \log(\max(t,1))$. We apply Proposition \ref{proposition:VC-subgraph_bound} with $p=2$ and obtain that
\begin{equation*}
  \int_{0}^{\delta} \sqrt{\log(N(\mathcal{F}, e_{n,2}, \epsilon))} \, d\epsilon
\end{equation*}    
is bounded above by
\begin{equation*}
  \delta \cdot \biggl( \sqrt{\log(c_{\mathbf{v}})} + \sqrt{2 \mathbf{v} \log(2)} + \sqrt{\mathbf{v} \log(S_{n}/n)} \biggr) + \sqrt{2\mathbf{v}} \int_{0}^{\delta} \sqrt{\log_{+}(2M/\epsilon)} \, d\epsilon.
\end{equation*}
Using that the last term is finite for all $\delta>0$ and $\bE[\sqrt{\log(S_{n}/n)}] \leq \sqrt{\bE[S_{n}/n]} = \sqrt{\bE[L_{1}]}$ we obtain \eqref{assumption_uniform_CLT}.
\end{proof}

\noindent We now turn to the proof of Theorem \ref{theorem:uniform_rates_of_convergence}.

\begin{proof}[Proof of Theorem \ref{theorem:uniform_rates_of_convergence}] \label{subsection:proof_of_theorem_rates_of_convergence}
  Using Lemma \ref{lemma:probability_inequality} and $\bE[L_{1}^{2}] \geq \sigma^{2}$ we obtain that for all $n \geq 8 \bE[L_{1}^{2}]/\epsilon^{2}$
\begin{equation} \label{proof_uniform_rates_of_convergence}
  \bP(\norm{\mu_{n}-\mu} \geq \epsilon) \leq 4 \cdot \bP(\norm{\mu_{\xi,n}} \geq \epsilon/4).
\end{equation}
The exponential inequality for independent subgaussian random variables yields that for $D \in \mathcal{D}$ and $t \geq 0$
\begin{equation*}
  \bP_{\xi}( \abs{\sum_{i=1}^{n} \xi_{i} Y_{i}(\mathbf{I}_{D})} \geq t) \leq 2 \cdot \exp \biggl( -\frac{t^{2}}{2 \sum_{i=1}^{n} Y_{i}^{2}(\mathbf{I}_{D})} \biggr).
\end{equation*}
Using this with $t=n \epsilon/4$ and $\sum_{i=1}^{n} Y_{i}^{2}(\mathbf{I}_{D}) \leq S_{n,2}$ we obtain
\begin{equation*}
  \bP_{\xi}( \abs{\mu_{\xi,n}(\mathbf{I}_{D})} \geq \epsilon/4) \leq 2 \cdot \exp \biggl( -\frac{\epsilon^{2}}{2^{5}} \cdot \frac{ n^{2}}{S_{n,2}} \biggr).
\end{equation*}
Using \ref{H3} we see that as $D \in \mathcal{D}$ varies, $\{ \mathbf{I}_{D}(X_{i,j}) \, : \, i=1,\dots,n,  \; j=1,\dots,L_{i} \}$ gives at most $m^{\mathcal{D}}(S_{n})$ different sets. We deduce that
\begin{equation*}
  \bP_{\xi}(\norm{\mu_{\xi,n}} \geq \epsilon/4) \leq 2 \cdot m^{\mathcal{D}}(S_{n}) \cdot \exp \biggl( -\frac{\epsilon^{2}}{2^{5}} \cdot \frac{n^{2}}{S_{n,2}} \biggr).
\end{equation*}
By taking the expectation on both sides and using \eqref{proof_uniform_rates_of_convergence} we obtain \eqref{uniform_rates_of_convergence}. Turning to \eqref{uniform_rates_of_convergence_2}, we have
\begin{align*}
  \bP(\norm{\mu_{n}-\mu} \geq \epsilon) &\leq \bP(\norm{\mu_{n}-\mu} \geq \epsilon, \, S_{n} \leq \alpha n, \, S_{n,2} \leq \beta n) \\
  &+ \bP(S_{n} > \alpha n) + \bP(S_{n,2} > \beta n).
\end{align*}  
Using $m^{\mathcal{D}}(S_{n}) \leq 2 S_{n}^{\mathbf{v}-1}$ in the proof of \eqref{uniform_rates_of_convergence} we obtain
\begin{equation*}
  \bP(\norm{\mu_{n}-\mu} \geq \epsilon, \, S_{n} \leq \alpha n, \, S_{n,2} \leq \beta n) \leq 16 \cdot (\alpha n)^{\mathbf{v}-1} \cdot \exp \biggl( -\frac{\epsilon^{2}}{2^{5}} \cdot \frac{n}{\beta} \biggr).
\end{equation*}
\end{proof}

\section{Proofs of other results} \label{section:proofs_of_other_results}

\noindent Before we prove Proposition \ref{proposition:halfspace_depth} we recall two basic properties of half-space depth.

\begin{lemma}[\citet{Masse-2004}, Proposition 4.5] \label{lemma:continuity_halfspace_depth}
  Assume \ref{H4}. The following holds: \\
(i) The function $(x,u) \mapsto \mu(H_{x,u})$ is continuous on $\mathbb{R}^{d} \times S^{d-1}$. \\
(ii) The function $x \mapsto D(x,\mu)$ is continuous.
\end{lemma}

\begin{lemma}[\citet{Zuo-2000a}, Theorem 2.1] \label{lemma:vanishing_at_infinity}
The following holds:
\begin{equation*}  
  \lim_{r \to \infty} \sup_{x \in \mathbb{R}^{d} : \norm{x} \geq r} D(x,\mu) = 0.
\end{equation*}  
\end{lemma}

\begin{proof}[Proof of Proposition \ref{proposition:halfspace_depth}]
  We first notice that the class of functions $\mathcal{F} = \{ \mathbf{I}_{H} \, : \, H \in \mathcal{H} \}$ is uniformly bounded, non-empty, measurable, and has VC-dimension $\mathbf{v}=d+1$. Hence, $\mathcal{F}$ satisfies assumptions \ref{H2} and \ref{H3}. Since
\begin{equation*}  
  \sup_{x \in \mathbb{R}^{d}} \abs{D(x,\mu) - D(x,\mu_{n})} \leq \norm{\mu-\mu_{n}},
\end{equation*}
(i) and (iii) follow from Corollary \ref{corollary:uniform_LLN} and Theorem \ref{theorem:uniform_rates_of_convergence}, respectively. We now turn to the proof of (ii). Using $D(x,\mu) = \mu(H_{x,u_{x}})$ and $D(x,\mu_{n}) = \mu_{n}(H_{x,u_{x,n}})$ for some direction $u_{x,n}(\omega) \in S^{d-1}$ depending on $\omega \in \Omega_{Y}$, we see that
\begin{equation} \label{clt_halfspace_depth}
\begin{aligned}
  \sqrt{n} (\mu_{n}(H_{x,u_{x,n}}) - \mu(H_{x,u_{x,n}})) &\leq \sqrt{n} (D(x,\mu_{n}) - D(x,\mu)) \\
  &\leq \sqrt{n} (\mu_{n}(H_{x,u_{x}}) - \mu(H_{x,u_{x}})).
\end{aligned}
\end{equation}
We let $\mathcal{G} \coloneqq \{ \mathbf{I}_{H_{x,u_{x}}} \, : \, x \in A \}$ and notice that $\mathcal{G}$ and $\mathcal{G}^{\prime}_{\delta,p}$, where $\delta \in [0,\infty]$ and $p \geq 1$, are all measurable in the sense of Definition \ref{definition:measurable_class}. It follows from Corollary \ref{corollary:uniform_CLT} that
\begin{equation*}
  \sqrt{n} (\mu_{n} - \mu) \xrightarrow[]{d} W \text{ in } \ell_{\infty}(\mathcal{A}),
\end{equation*}
where we identify the class $\mathcal{G}$ with the set $A$. In view of \eqref{clt_halfspace_depth} it is enough to show that for all $\epsilon>0$
\begin{equation*}
  \lim_{n \to \infty} \bP^{\ast}(\sup_{x \in A} \abs{W_{n}(H_{x,u_{x}}) - W_{n}(H_{x,u_{x,n}})} \geq \epsilon) \leq \epsilon,
\end{equation*}
where $W_{n}=\sqrt{n}(\mu_{n}-\mu)$ and we write $W_{n}(H)$ for $W_{n}(\mathbf{I}_{H})$. The above LHS is bounded above by
\begin{equation*}
  \bP(\sup_{e_{\mu}(H_{1},H_{2}) \leq \frac{1}{k}} \abs{W_{n}(H_{1}) -W_{n}(H_{2})} \geq \epsilon ) + \bP^{\ast}( \sup_{x \in A} e_{\mu}(H_{x,u_{x}}, H_{x,u_{x,n}}) > \frac{1}{k} ),
\end{equation*}
where, as before, we write $e_{\mu}(H_{1},H_{2})$ for $e_{\mu}(\mathbf{I}_{H_{1}},\mathbf{I}_{H_{2}})$. The asymptotic equicontinuity condition \eqref{asymptotic_equicontinuity} yields that
\begin{equation*}
  \lim_{k \to \infty} \limsup_{n \to \infty} \bP(\sup_{e_{\mu}(H_{1},H_{2}) \leq \frac{1}{k}} \abs{W_{n}(H_{1}) -W_{n}(H_{2})} \geq \epsilon ) = 0.
\end{equation*}
Therefore, for $k$ large enough
\begin{equation*}
  \limsup_{n \to \infty} \bP(\sup_{e_{\mu}(H_{1},H_{2}) \leq \frac{1}{k}} \abs{W_{n}(H_{1}) -W_{n}(H_{2})} \geq \epsilon ) \leq \frac{\epsilon}{2}.
\end{equation*}
Since $e_{\mu}^{2}(H_{x,u_{x}}, H_{x,u_{x,n}}) \leq \mu(H_{x,u_{x}} \Delta H_{x,u_{x,n}})$, where $C \Delta D$ is the symmetric difference of sets $C$ and $D$, it suffices to show that
\begin{equation} \label{proof_halfspace_depth_0}
  \lim_{n \to \infty} \bP^{\ast}( \sup_{x \in A} \mu(H_{x,u_{x}} \Delta H_{x,u_{x,n}}) > \frac{1}{k^{2}} ) \leq \frac{\epsilon}{2}.
\end{equation}
For $0<\alpha \leq \frac{1}{4k^{2}}$ consider the upper level set $Q_{\alpha}(\mu) \coloneqq \{ x \in \mathbb{R}^{d} : D(x,\mu) \geq \alpha \}$. Using that $\mathcal{F}$ is Glivenko-Cantelli and Egoroff's theorem, let $B \in \Sigma_{Y}$ with $\bP(B) \geq 1-\frac{\epsilon}{2}$ such that $\norm{\mu-\mu_{n}} \to 0$ uniformly on $B$. In particular, $\norm{\mu-\mu_{n}} \leq \alpha$ for all large $n$. It follows that for $\omega \in B$ and large $n$
\begin{equation} \label{proof_halfspace_depth_1}
  \sup_{x \in A \setminus Q_{\alpha}(\mu)} \mu(H_{x,u_{x}} \Delta H_{x,u_{x,n}(\omega)}) \leq \sup_{x \in A \setminus Q_{\alpha}(\mu)} \hspace{-0.3cm} \mu(H_{x,u_{x}}) + \sup_{x \in A \setminus Q_{\alpha}(\mu)} \hspace{-0.3cm} \mu(H_{x,u_{x,n}(\omega)}) \leq 4 \alpha.
\end{equation}
We show below that for $\omega \in B$
\begin{equation} \label{proof_halfspace_depth_2}
  \lim_{n \to \infty} \sup_{x \in A \cap Q_{\alpha}(\mu)} \mu(H_{x,u_{x}} \Delta H_{x,u_{x,n}(\omega)}) = 0.
\end{equation}
From \eqref{proof_halfspace_depth_1} and \eqref{proof_halfspace_depth_2}, we obtain \eqref{proof_halfspace_depth_0}. We are left to show \eqref{proof_halfspace_depth_2}. If \eqref{proof_halfspace_depth_2} does not hold, there exist $\delta>0$, $\omega_{k} \in B$, $n_{k}$, and $x_{k} \in A \cap Q_{\alpha}(\mu)$ such that
\begin{equation} \label{proof_halfspace_depth_3}
  \mu(H_{x_{k},u_{x_{k}}} \Delta H_{x_{k},u_{x_{k},n_{k}}(\omega_{k})}) \geq \delta.
\end{equation}
Since $A \cap Q_{\alpha}(\mu)$ (by Lemmas \ref{lemma:continuity_halfspace_depth}-\ref{lemma:vanishing_at_infinity}) and $S^{d-1}$ are compact, there exist subsequences $\{ x_{k_{j}} \}_{j=1}^{\infty}$ and $\{ u_{x_{k_{j}},n_{k_{j}}(\omega_{k_{j}})} \}_{j=1}^{\infty}$ converging to some $x_{0} \in A \cap Q_{\alpha}(\mu)$ and $u_{0} \in S^{d-1}$, respectively. Using Lemma \ref{lemma:continuity_halfspace_depth} and $\norm{\mu-\mu_{n}} \to 0$ uniformly on $B$, we obtain that 
\begin{equation*}
  \mu(H_{x_{0},u_{0}}) = \lim_{j \to \infty} \mu(H_{x_{k_{j}},u_{x_{k_{j}},n_{k_{j}}(\omega_{k_{j}})}}) = \lim_{j \to \infty} D_{n_{k_{j}}}(x_{k_{j}},\mu)(\omega_{k_{j}}) = D(x_{0},\mu).
\end{equation*}
As $x_{0} \in A$ we deduce that $u_{0}=u_{x_{0}}$. Using the uniqueness of the minimal direction and the continuity of $D(\cdot,\mu)$, we see that $\{ u_{x_{k_{j}}} \}_{j=1}^{\infty}$ also converges to $u_{x_{0}}$. It follows from \eqref{H4} that
\begin{equation*}
  \lim_{j \to \infty} \mu(H_{x_{k_{j}},u_{x_{k_{j}}}} \Delta H_{x_{k_{j}},u_{x_{k_{j}},n_{k_{j}}}(\omega_{k_{j}})}) = 0.
\end{equation*}
Thus, \eqref{proof_halfspace_depth_3} fails for $k=k_{j}$ and large $j$ concluding the proof.
\end{proof}

We now turn to the proofs of the results of Subsection \ref{subsection:tree-indexed_random_variables}.

\begin{proof}[Lemma \ref{lemma:independent_and_identically_distributed}]
  We begin by showing that $Y_{v_{i}}$ has the same distribution as $Y_{\emptyset}$. Using that $\{ Y_{v} \}_{v \in \mathbb{V}}$ are identically distributed, we obtain that for all $A \in C_{\infty}(\mathcal{F})$
\begin{equation*}
  \bP(Y_{v_{i}} \in A) = \sum_{w \in \mathbb{V}} P(Y_{v_{i}} \in A \, | \, v_{i}=w) \cdot \bP(v_{i}=w) = \bP(Y_{\emptyset} \in A).
\end{equation*}
Turning to independence, let $\Sigma_{n}$ be the $\sigma$-algebra generated by $Y_{v_{1}}, \dots, Y_{v_{n}}$. Using that $v_{1}, \dots, v_{n}$ are $\Sigma_{n-1}$-measurable and that, conditionally on $v_{n}=w$, $Y_{v_{n}}$ is independent of $\Sigma_{n-1}$, we obtain that for all $n \in \mathbb{N}$ and $A_{1}, \dots, A_{n} \in C_{\infty}(\mathcal{F})$
\begin{align*}
   \bP(\cap_{i=1}^{n} \{ Y_{v_{i}} \in A_{i} \} \, | \, v_{n}=w) =&\bE[\bP(\cap_{i=1}^{n} \{ Y_{v_{i}} \in A_{i} \} \, | \, v_{n}=w, \Sigma_{n-1})] \\
  =& \bP(\cap_{i=1}^{n-1} \{ Y_{v_{i}} \in A_{i} \} \, | \, v_{n}=w) \cdot \bP(Y_{\emptyset} \in A_{n}).
\end{align*}
Multiplying both sides by $\bP(v_{n}=w)$ and summing over $w \in \mathbb{V}$, we conclude that
\begin{equation*}
  \bP( \cap_{i=1}^{n} \{ Y_{v_{i}} \in A_{i} \}) = \bP(\cap_{i=1}^{n-1} \{ Y_{v_{i}} \in A_{i} \} ) \cdot \bP(Y_{\emptyset} \in A_{n}).
\end{equation*}
\end{proof}  

\begin{proof}[Proposition \ref{proposition:Lotka-Nagaev-estimator}]
  We first notice that, as $\bP(L_{1} \geq 1)=1$ and $\bE[L_{1}]>1$, $\abs{V_{n}}$ is non-decreasing and diverges to infinity at rate $\bE[L_{1}]^{n}$. For all $\epsilon>0$ we have that
\begin{align*}
  \bP( \sup_{j \geq n} \norm{\hat{\mu}_{j} - \mu} \geq \epsilon) &\leq \bP( \sup_{j \geq n} \norm{\hat{\mu}_{j} - \mu} \geq \epsilon \, | \, \abs{V_{n}} \geq n) \cdot \bP(\abs{V_{n}} \geq n) + \bP(\abs{V_{n}} < n) \\
  &\leq \bP( \sup_{j \geq n} \norm{\mu_{j} - \mu} \geq \epsilon) + \bP(\abs{V_{n}} < n).
\end{align*}
Using Corollary \ref{corollary:uniform_LLN} and $\lim_{n \to \infty} \bP(\abs{V_{n}} < n) = 0$ we obtain that
\begin{equation*}
  \lim_{n \to \infty} \bP( \sup_{j \geq n} \norm{\hat{\mu}_{j} - \mu} \geq \epsilon) = 0,
\end{equation*}
which gives (i). Turning to (ii), let $H: \ell_{\infty}(\mathcal{F}) \to \mathbb{R}$ be a bounded and continuous function. We show below that for all $\epsilon>0$
\begin{equation} \label{Lotka-Nagaev_uniform_CLT}
  \limsup_{j \to \infty} \abs{\bE^{\ast}[H(\hat{W}_{j})] - \bE[H(W)]} \leq \epsilon.
\end{equation}
Using Corollary \ref{corollary:uniform_CLT} let $k_{0}$ such that $\abs{\bE^{\ast}[H(W_{k})] - \bE[H(W)]} \leq \epsilon$ for all $k \geq k_{0}$. It holds that
\begin{align*}
  \abs{\bE^{\ast}[H(\hat{W}_{j})] - \bE[H(W)]} &\leq \sum_{k=1}^{\infty} \abs{\bE^{\ast}[H(\hat{W}_{j}) \, | \, \abs{V_{j}} = k] - \bE[H(W)]} \cdot \bP( \abs{V_{j}} = k) \\
  &\leq \sum_{k=1}^{k_{0}-1} \abs{\bE^{\ast}[H(W_{k})] - \bE[H(W)]} \cdot \bP( \abs{V_{j}} = k) + \epsilon.
\end{align*}
Using that $H$ is bounded and $\lim_{j \to \infty} \bP( \abs{V_{j}} = k) = 0$ we obtain \eqref{Lotka-Nagaev_uniform_CLT}.
\end{proof}

\begin{proof}[Proposition \ref{proposition:Harris-estimator}]
Let $T_{j} \coloneqq \sum_{l=0}^{j} \abs{V_{l}}$. If $L_{1}=1$ then $T_{j} = j+1$ and (i) and (ii) follow from Corollaries \ref{corollary:uniform_LLN} and \ref{corollary:uniform_CLT}, respectively. We assume in the following that $\bE[L_{1}] > 1$. In this case, a standard convergence result for branching processes gives that
\begin{equation} \label{convergence_summation_V_j}
  \frac{T_{j}}{\bE[L_{1}]^{j}} \xrightarrow[]{a.s.} \frac{\bE[L_{1}]}{\bE[L_{1}]-1} \cdot \overline{W},
\end{equation}    
where $\overline{W}$ is a positive random variable. Using that $T_{j}$ is increasing and Corollary \ref{corollary:uniform_LLN}, we obtain that for all $\epsilon>0$
\begin{equation*}
  \lim_{n \to \infty} \bP( \sup_{j \geq n} \norm{\tilde{\mu}_{j} - \mu} \geq \epsilon) = \lim_{n \to \infty} \bP( \sup_{j \geq n} \norm{\mu_{j} - \mu} \geq \epsilon) = 0,
\end{equation*}
which gives (i). Turning to (ii), as in the proof of Theorem \ref{theorem:uniform_CLT}, we show that (a) the finite dimensional distributions of $\tilde{W}_{j}$ converge in law, (b) the space $(\mathcal{F},e_{\mu})$ is totally bounded, and (c) the process $\tilde{W}_{j}$ is asymptotically equicontinuous. Now, the convergence of finite dimensional distributions follows from \eqref{multivariate_CLT} with $n$ replaced by $T_{j}$ and the multivariate version of Anscombe central limit theorem \citep{Gleser-1969} using the convergence in \eqref{convergence_summation_V_j} to a positive random variable \citep{Blum-1963}. Turning to (b) and (c), we notice that \eqref{H3} implies both \eqref{assumption_uniform_CLT} with deterministic $n$ and \eqref{assumption_uniform_CLT} with $n$ replaced by $T_{j}$ (cfr.\ the proof of Corollary \ref{corollary:uniform_CLT}). In this second case, we notice that in the proof of Corollary \ref{corollary:uniform_CLT}
\begin{equation*}
  \bE[ S_{T_{j}}/T_{j} ] = \bE[\tilde{\mu}_{j}(1)] \leq 1+\bE[L_{1}].
\end{equation*}
Following the proof of Theorem \ref{theorem:uniform_CLT} we see that \eqref{assumption_uniform_CLT} with deterministic $n$ implies that $(\mathcal{F},e_{\mu})$ is totally bounded and \eqref{assumption_uniform_CLT} with $n$ replaced by $T_{j}$ implies that $\tilde{W}_{j}$ is asymptotically equicontinuous.
\end{proof}

\begin{proof}[Proof of Propositon \ref{proposition:Lotka-Nagaev-estimator_joint_convergence}]
  Let $\epsilon > 0$. It is enough to show that
\begin{equation} \label{Lotka-Nagaev_joint_CLT}
  \limsup_{j \to \infty} \abs{\bE[H(\hat{W}_{j+1}, \dots, \hat{W}_{j+s})] - \bE[H(W^{1},\dots,W^{s})]} \leq \epsilon
\end{equation}
for all continuous function $H:(\ell^{\infty}(\mathcal{F}))^{s} \to \mathbb{R}$ with compact support (see Exercise 5.10 of \citet{Billingsley-2013}). Using Corollary \ref{corollary:uniform_CLT} and the properties of $H$, let $k_{0}$ such that for all $k \geq k_{0}$, $l=1,\dots,s$ and $t_{1},\dots,t_{s} \in \ell^{\infty}(\mathcal{F})$
\begin{equation*}
  \abs{ \bE[H(t_{1},\dots,t_{l-1},W_{k}^{l},t_{l+1}, \dots, t_{s})] - \bE[H(t_{1},\dots,t_{l-1},W^{l},t_{l+1}, \dots, t_{s})] } \leq \frac{\epsilon}{s},
\end{equation*}
where $\{W_{k}^{l}\}_{l=1}^{s}$ are independent copies of $W_{k}$. Let
\begin{equation*}
  H_{l,k} \coloneqq \bE[H(\hat{W}_{j+1},\dots,\hat{W}_{j+l},W^{l+1}, \dots, W^{s}) \, | \, \abs{V_{j+l-1}}=k].
\end{equation*}  
By conditioning on $\hat{W}_{j+1}, \dots, \hat{W}_{j+l-1}$ and $W^{l+1}, \dots, W^{s}$, and using that, conditionally on $\abs{V_{j+l-1}}=k$, $\hat{W}_{j+l}$ has the same distribution as $W_{k}^{l}$, we obtain that $\abs{H_{l,k}-H_{l-1,k}} \leq \epsilon/s$ for all $k \geq k_{0}$. We deduce that
\begin{equation*}
  \abs{\bE[H(\hat{W}_{j+1}, \dots, \hat{W}_{j+s})] - \bE[H(W^{1},\dots,W^{s})]} 
\end{equation*}
is bounded above by
\begin{equation*}
  \sum_{l=1}^{s} \sum_{k=1}^{k_{0}-1} \abs{H_{l,k}-H_{l-1,k}} \cdot \bP(\abs{V_{j+l-1}}=k) + \epsilon.
\end{equation*}
We obtain \eqref{Lotka-Nagaev_joint_CLT} using that $H$ is bounded and $\lim_{j \to \infty} \bP(\abs{V_{j+l-1}}=k) = 0$.
\end{proof}    

\section{Concluding remarks} \label{section:concluding_remarks}

In this paper, we establish functional limit laws for the intensity measure of a point process, allowing for inquiry into depth functions for tree-indexed random variables and other point processes. We observe that the assumption of uniform boundedness in Theorems \ref{theorem:uniform_LLN} and \ref{theorem:uniform_CLT} can be replaced by appropriate moment and measurability conditions at the cost of a longer proof. Also, the uniqueness of the minimal half-space in Proposition \ref{proposition:halfspace_depth} can be relaxed as in \citet{Masse-2004}. Extensions of these ideas to other depth functions are interesting open problems.

Turning to branching random walk, our results provide uniform consistency and asymptotic normality for the Laplace transform of the point process for bounded displacements. Extensions to more general displacements under weaker moment conditions and the behavior under other sampling schemes will be investigated elsewhere.

\medspace

\medspace

\address

\end{document}